\documentclass[12pt]{article}
\usepackage{amssymb,amsmath,amsthm,amscd,tikz}

\title{Topological Classification of Multiaxial $U(n)$-Actions\\
(with an appendix by Jared Bass)}
\author{Sylvain Cappell\thanks{Research was partially supported by an NSF grant} \\
Courant Institute, New York University \\
\\
Shmuel Weinberger\thanks{Research was partially supported by an NSF grant} \\
University of Chicago \\
\\
Min Yan\thanks{Research was supported by Hong Kong Research Grant Council General Research Fund 604408 and 605610} \\ 
Hong Kong University of Science and Technology}

\DeclareMathOperator{\loc}{loc}
\DeclareMathOperator{\rel}{rel}

\DeclareMathOperator{\iif}{if}

\newcommand{\sub}{\subset}
\newcommand{\pa}{\partial}

\newcommand{\bb}{\mathbb}

\newtheorem{theorem}{Theorem}[section]
\newtheorem{lemma}[theorem]{Lemma}
\newtheorem{corollary}[theorem]{Corollary}
\newtheorem{proposition}[theorem]{Proposition}

\theoremstyle{definition}

\newtheorem*{definition*}{Definition}

\begin{document}
\maketitle

\section{Introduction}

In the last half century great progress has been made on both the differentiable and topological classification of finite group actions on spheres and more general manifolds. Deep, albeit indirect, connections of transformation groups to representation theory were discovered. For positive dimensional groups beyond the case of the circle, essentially the only classification results obtained for differentiable actions are the classical results of M. Davis and W. C. Hsiang \cite{dh} and their further development with J. Morgan \cite{dhm} on concordance classes of multiaxial actions on homotopy spheres, in certain dimension ranges. 

On the other hand, certain topological phenomena, such as  periodicity \cite{wy1} and the replacement of fixed points \cite{cw,cwy} showed that the topological classification of actions of positive dimensional groups must be very different from the smooth case. The present paper begins the classification of topological actions on manifolds by positive dimensional groups beyond the case of the circle, by obtaining general results on multiaxial actions on topological manifolds. Here we will work with a more flexible notion of multiaxial (and without the dimension conditions) than had been considered for smooth actions. An action of a unitary group $U(n)$ on a manifold will be called multiaxial if all of its isotropy subgroups are unitary subgroups, and the corresponding interiors of strata are locally flat submanifolds. Our results will show that topological multiaxial actions are far more profuse and their classification is quite different from the smooth case, even when restricted to spheres. For example, the homology of Grassmannians enters into the classification of topological actions homotopically modeled on multiaxial representation spheres. 

The connection of the theory of topological actions to representation theory is less direct than for smooth actions. This reflects the failure in the topological setting of some of the basic building blocks of the analogous smooth or PL theory of finite group actions. Whitehead torsion, a cornerstone of the classical theory of lens spaces, plays in the general topological category a diminished role because of the absence of (canonical) tubular neighborhoods around fixed points \cite{quinn,st} and more generally around subsets of given orbit types. Indeed, Milnor's counterexamples to the Hauptvermutung showed that classical Whitehead torsion is not even always definable in the topological category for non-free actions. The divergence for actions of finite groups of the topological classification from the smooth or PL ones was strikingly reflected in the existence of non-linear similarities between some linearly inequivalent representations \cite{cs1}.Ê

On the other hand, key invariants defined in smooth or PL settings using the equivariant signature operator do remain well defined in topological settings \cite{csw,hp,mr1} and play a major role there.Ê

In this introduction, for simplicity and ease of exposition, we make the stronger assumption that $G=U(n)$ acts locally smoothly (thus also assuring local flatness of fixed sets). In other words, every orbit has a neighborhood equivariantly homeomorphic to an open subset of an orthogonal representation of $G$. Moreover, we concentrate on the classical and more restrictive notion of multiaxial actions, for which the orthogonal representations are of the form $k\rho_n\oplus j \epsilon$, where $\rho_n$ is the defining representation of $U(n)$ on ${\bb C}^n$ and $\epsilon$ is the trivial representation. While this allows for different choices of $k$ and $j$ at different locations in a manifold, the results presented in the introduction will assume the same $k$ and $j$ everywhere. In such a setting, we say the manifold is modeled on $k\rho_n\oplus j\epsilon$. Examples are the representation $k\rho_n\oplus j \epsilon$ and the associated representation sphere. 

The isotropy subgroups of a multiaxial $U(n)$-manifold $M$ are conjugate to the specific unitary subgroups $U(i)$ of $U(n)$ that fixes the subspace $0\oplus{\bb C}^{n-i}$ of ${\bb C}^n$. Then $M$ is stratified by $M_{-i}=U(n)M^{U(i)}$, the set of points fixed by some conjugate of $U(i)$. Correspondingly, the orbit space $X=M/U(n)$ is stratified by $X_{-i}=M^{U(i)}/U(n-i)$. 

Our goal is to study the isovariant structure set $S_{U(n)}(M)$. Classically, the structure set $S(X)$ of a compact topological manifold $X$ is the homeomorphism classes of topological manifolds equipped with a simple homotopy equivalence to $X$ (with homotopy and homeomorphism defining the equivalence relation). The notion can be extended to the setting of a $G$-manifold $M$ by letting $S_G(M)$ denote the equivariant homeomorphism classes of $G$-manifolds each equipped with an isovariant simple homotopy equivalence to $M$. It can also be extended to $S(X)$ for stratified spaces $X$ and stratified simple homotopy equivalences. We have $S_G(M)=S(M/G)$ when the orbit space $M/G$ is homotopically stratified \cite{we1}.

Classical surgery theory formulates $S(X)$ initially in terms of $s$-cobordism classes and then employs the $s$-cobordism theorem to reformulate this in terms of the more meaningful homeomorphism classification. The topological isovariant surgery theory of \cite{we1} similarly employs the stratified (and thus the equivariant) $s$-cobordism theorem of Quinn \cite{quinn} and of Steinberger \cite{st}\footnote{It is in fact more convenient to study the classification of topological manifolds in the slightly larger (by at most a ${\bb Z}$ summand \cite{quinn2,quinn3}) setting of $s$-cobordism classes of homology manifolds. Similar considerations apply to $S_G(M)$. The results of the present paper are stated in this slightly larger setting.}.

Let $X_{\alpha}$ be the strata of a stratified space $X$. The pure strata 
\[
X^{\alpha}=X_{\alpha}-X_{<\alpha},\quad X_{<\alpha}=\cup_{X_{\beta}\subsetneq X_{\alpha}}X_{\beta}
\]
are generally noncompact manifolds, and we have natural restriction maps
\[
S(X)\to \oplus S^{\text{proper}}(X^{\alpha}).
\]
Here $S^{\text{proper}}$ denotes the proper homotopy equivalence version of the structure set. If we further know that all pure strata of links between strata of $X$ are connected and simply connected (or more generally, the fundamental groups of these strata have trivial $K$-theory in low dimensions, according to Quinn \cite{quinn}), then the complement $\bar{X}^{\alpha}$ of (the interior of) a regular neighborhood of $X_{<\alpha}$ in $X_{\alpha}$ is a topological manifold with boundary $\pa \bar{X}^{\alpha}$ and interior $X^{\alpha}$, and the restriction maps natually factor through the structures of $(\bar{X}^{\alpha},\pa \bar{X}^{\alpha})$
\[
S(X)\to \oplus S(\bar{X}^{\alpha},\pa \bar{X}^{\alpha})
\to \oplus S^{\text{proper}}(X^{\alpha}).
\]
The difference between the simple homotopy structure of $(\bar{X}^{\alpha},\pa \bar{X}^{\alpha})$ and the proper homotopy structure of $X^{\alpha}$ is captured by the finiteness obstruction at infinity and related Whitehead torsion considerations. 

The pure strata of links are indeed connected and simply connected for multiaxial $U(n)$-manifolds. Our main result states that the stratified simple homotopy structure set of $X=M/U(n)$ is almost always determined by the restrictions to $S(\bar{X}^{-i},\pa \bar{X}^{-i})$ using a particular half of the set of strata $X_{-i}$. More general versions are given by Theorems \ref{ssplit}, \ref{ssplit3}, \ref{ssplit2}.

\begin{theorem}\label{mainth1}
Suppose $M$ is a multiaxial $U(n)$-manifold modeled on $k\rho_n\oplus j\epsilon$, and $X=M/U(n)$ is the orbit space. 
\begin{enumerate}
\item If $k\ge n$ and $k-n$ is even, then we have natural splitting
\[
S_{U(n)}(M)
=\oplus_{i\ge 0}S(\bar{X}^{-2i},\pa\bar{X}^{-2i})
=\oplus_{i\ge 0}S^{\text{\rm alg}}(X_{-2i},X_{-2i-1}).
\]
\item If $k\ge n$, $k-n$ is odd and $M=W^{U(1)}$ for a multiaxial $U(n+1)$-manifold modeled on $k\rho_{n+1}\oplus j\epsilon$, then we have natural splitting
\[
S_{U(n)}(M)
=S^{\text{\rm alg}}(X)\oplus\left(\oplus_{i\ge 1}S(\bar{X}^{-2i+1},\pa\bar{X}^{-2i+1})\right)
=S^{\text{\rm alg}}(X)\oplus\left(\oplus_{i\ge 1}S^{\text{\rm alg}}(X_{-2i+1},X_{-2i})\right).
\]
\item If $k\le n$, then $S_{U(n)}(M)=S_{U(k)}(M^{U(n-k)})$. Since $M^{U(n-k)}$ is a multiaxial $U(k)$-manifold modeled on $k\rho_k\oplus j\epsilon$, this case is reduced to $k=n$ treated in part 1. 
\end{enumerate}
\end{theorem}

The condition $k\le n$ was always assumed in the results of \cite{davis,dh,dhm} on differentiable actions. For the reduction to the case $k=n$, see Lemma \ref{hereditary} and the subsequent discussion.

The {\em algebraic structure set} $S^{\text{alg}}$ in the theorem denotes the following familiar homotopy functor \cite{ra1}.

\begin{definition*}
For any (reasonable) topological space $X$, let ${\bb S}^{\text{alg}}(X)$ be the homotopy fibre of the surgery assembly map ${\bb H}_*(X;{\bb L})\to {\bb L}(\pi_1X)$. Then $S^{\text{alg}}(X)=\pi_{\dim X}{\bb S}^{\text{alg}}(X)$.
\end{definition*}

In the definition, ${\bb L}(\pi)$ is the (simple) surgery obstruction spectrum for the fundamental group $\pi$, and ${\bb H}_*(X;{\bb L})$ is the homology theory associated to the spectrum ${\bb L}={\bb L}(e)$. If $X$ is a topological manifold of dimension $\ge 5$ (or dimension $4$ in case $\pi_1X$ is not too bad \cite{freedman}), then $S^{\text{alg}}(X)$ is the usual structure set that classifies topological (in fact, homological) manifolds simple homotopy equivalent to $X$. For a general topological space $X$, however, $S^{\text{alg}}(X)$ no longer carries that geometrical meaning and is for the present purpose the result of some algebraic computation.

Notice that the expression in terms of $S^{\text{\rm alg}}(X_{-i},X_{-i-1})$ involves only objects that are a priori associated to the group action.  However, the map from the left hand side to the right hand side, while related to the forgetful map to $S(\bar{X}^{-i},\pa\bar{X}^{-i-1})$, is not quite obvious to define.

For a taste of what to expect when $k$ and $j$ are not assumed constant, the following is the simplest case of Theorem \ref{ssplit3}. The proof is given at the end of Section \ref{structure}.

\begin{theorem}\label{mainth2}
Suppose the circle $S^1$ acts semifreely and locally linearly on a topological manifold $M$, such that the fixed points $M^{S^1}$ is a locally flat submanifold. Let $M_0^{S^1}$ and $M_2^{S^1}$ be the unions of those connected components of $M^{S^1}$ that are, respectively, of codimensions $0$ mod $4$ and $2$ mod $4$. Let $N$ be the complement of (the interior of) an equivariant tube neighborhood of $M^{S^1}$, with boundaries $\pa_0N$ and $\pa_2N$ corresponding to the two parts of the fixed points. Then
\[
S_{S^1}(M)=S(M_0^{S^1})\oplus S(N/S^1,\pa_2N/S^1,\rel \pa_0N/S^1).
\]
\end{theorem}

We note that $N/S^1$ is a manifold with boundary divided into two parts $\pa_0$ and $\pa_2$. The second summand means the homeomorphism classes of manifolds simple homotopy equivalent to $N/S^1$ that restricts to a simple homotopy equivalence on $\pa_2$ and a homeomorphism on $\pa_0$. We also note that it is a special feature of the circle action that the condition of the extendability of $M$ to a multiaxial $U(2)$-manifold is not needed.  It is an open question whether or not, in general, one can dispense with the extendability condition in part 2 of Theorem \ref{mainth1}.

For $k\ge n$, the terms $S(\bar{X}^{-i},\pa\bar{X}^{-i})$ in the decompositions of Theorem \ref{mainth1} could be reformulated in terms of the isovariant structure set
\[
S(\bar{X}^{-i},\pa\bar{X}^{-i})
=S_{U(n-i)}(M^{U(i)},\rel U(n-i)M^{U(i+2)}).
\]
Here $M^{U(i)}$ is actually a multiaxial $U(n-i)$-manifold modeled on $k\rho_{n-i}\oplus j\epsilon$, and $U(n-i)M^{U(i+2)}$ is the stratum of the multiaxial $U(n-i)$-manifold two levels down. The right side classifies those $U(n-i)$-manifolds isovariantly simple homotopy equivalent to $M^{U(i)}$, such that the restrictions to the stratum two levels down are already equivariantly homeomorphic. The decomposition in Theorem \ref{mainth1} is then equivalent to the decomposition
\[
S_{U(n)}(M)
=S_{U(n)}(M,\rel U(n)M^{U(i)})\oplus S_{U(n-i)}(M^{U(i)}),\text{ for } k-n+i\text{ even}.
\]
The map to the second summand is the obvious restriction. The fact that this restriction is onto has the following interpretation.

\begin{theorem}\label{mainth3}
Suppose $M$ is a multiaxial $U(n)$-manifold modeled on $k\rho_n\oplus j\epsilon$. Suppose $k\ge n>i$, $k-n+i$ is even, and additionally, when $k-n$ is odd, we have $M=W^{U(1)}$ for a multiaxial $U(n+1)$-manifold $W$ modeled on $k\rho_{n+1}\oplus j\epsilon$. Then for any $U(n-i)$-isovariant simple homotopy equivalence $\phi\colon V\to M^{U(i)}$, there is a $U(n)$-isovariant simple homotopy equivalence $f\colon N\to M$, such that $\phi=f^{U(i)}$ is the restriction of $f$. 
\end{theorem}

The theorem means that half of the fixed point subsets can be homotopically replaced. The homotopy replacement of the fixed point subset of the whole group has been studied in \cite{cw,cwy}. Here equivariant replacement is achieved for the fixed point subsets of certain proper subgroups (and not others); this is the first appearance of such a phenomenon. 

For $k\le n$, by $S_{U(n)}(M)=S_{U(k)}(M^{U(n-k)})$, we may apply Theorem \ref{mainth3} to the $k$-axial $U(k)$-manifold $M^{U(n-k)}$ and get the following homotopy replacement result: For any even $i\le k$ and $U(k-i)$-isovariant simple homotopy equivalence $\phi\colon V\to M^{U(n-k+i)}$, there is a $U(n)$-isovariant simple homotopy equivalence $f\colon N\to M$, such that $\phi=f^{U(n-k+i)}$ is the restriction of $f$. 

Algebraically, the terms $S^{\text{alg}}(X_{-i},X_{-i-1})$ and $S^{\text{alg}}(X)$ in the decompositions of Theorem \ref{mainth1} can be explicitly computed for the special case that $M$ is the unit sphere of the representation $k\rho_n\oplus j\epsilon$. For $k\ge n$, let $A_{n,k}$ be the number of Schubert cells of dimensions $0$ mod $4$ in the complex Grassmannian $G(n,k)$, and let $B_{n,k}$ be the number of cells of dimensions $2$ mod $4$. Specifically, $A_{n,k}$ is the number of $n$-tuples $(\mu_1,\dotsc,\mu_n)$ satisfying
\[
0\le \mu_1\le \dotsb \le \mu_n\le k-n,\quad
\sum\mu_i\text{ is even},
\]
and $B_{n,k}$ is the similar number for the case $\sum\mu_i$ is odd. Then the following computation is carried out in Section \ref{repsphere}.

\begin{theorem}\label{mainth4}
Suppose $S(k\rho_n\oplus j\epsilon)$ is the unit sphere of the representation $k\rho_n\oplus j\epsilon$, $k\ge n$. 
\begin{enumerate}
\item If $k-n$ is even, then we have
\[
S_{U(n)}(S(k\rho_n\oplus j\epsilon))
={\bb Z}^{\sum_{0\le 2i<n}A_{n-2i,k}}\oplus {\bb Z}_2^{\sum_{0\le 2i<n}B_{n-2i,k}},
\]
with the only exception that there is one less copy of ${\bb Z}$ in case $n$ is odd and $j=0$. 
\item If $k-n$ is odd, then we have
\[
S_{U(n)}(S(k\rho_n\oplus j\epsilon))
={\bb Z}^{A_{n,k-1}+\sum_{0\le 2i-1<n}A_{n-2i+1,k}}\oplus {\bb Z}_2^{B_{n,k-1}+\sum_{0\le 2i-1<n}B_{n-2i+1,k}},
\]
with the exceptions that there is one less copy of ${\bb Z}$ in case $n$ is even and $j=0$, and there is one more copy of ${\bb Z}_2$ in case $n$ is odd and $j>0$.
\end{enumerate}
\end{theorem}

The computation generalizes the classical computation for the fake complex projective space \cite[Section 14C]{wa2}. 

If $N$ is isovariant simple homotopy equivalent to the representation sphere $S(k\rho_n\oplus j\epsilon)$, then joining with the representation sphere $S(\rho_n)$ yields a manifold $N*S(\rho_n)$ isovariant simple homotopy equivalent to the representation sphere $S((k+1)\rho_n\oplus j\epsilon)$. This gives the suspension map
\[
*S(\rho_n)\colon
S_{U(n)}(S(k\rho_n\oplus j\epsilon))
\to S_{U(n)}(S((k+1)\rho_n\oplus j\epsilon)).
\]
A consequence of the calculation in Theorem \ref{mainth4} is the following, proved in Section \ref{suspend}.

\begin{theorem}\label{mainth5}
The suspension map is injective.
\end{theorem}

Finally, in Section \ref{quat}, we extend all the results to the similarly defined multiaxial $Sp(n)$-manifolds. 

For $k-n$ odd, the proofs of both Theorems \ref{mainth4} and \ref{mainth5} depend on very clever detailed calculations of the homology of certain orbit spaces (unlike the other calculations that depend on classical calculations of the cohomology of Grassmanians).  We would like to thank Jared Bass who wrote the appendix to this paper, presenting these calculations, taken from his forthcoming University of Chicago Ph.D. thesis.  We would also like to thank the Hebrew University of Jerusalem, and the University of Chicago for their hospitality during the work on this project.

\section{Strata of Multixial $U(n)$-Manifolds}

Smooth multiaxial manifolds were introduced and studied in \cite{davis, dh, dhm}, following earlier works on biaxial actions \cite{br1,br2,br3,hh1,ja}. As noted in the introduction, our definition of multiaxial actions in the topological category is more flexible and the actions are not assumed to be locally linear (just local flatness of strata), and the local model may vary at different parts of the manifold.

Let $U(n)$ be the unitary group of linear transformations of ${\bb C}^n$ preserving the Euclidean norm. By a {\em unitary subgroup}, we mean the subgroup of unitary transformations fixing a linear subspace of ${\bb C}^n$. If the fixed subspace has complex dimension $n-i$, then the unitary subgroup is conjugate to the specific unitary subgroup $U(i)$ of $U(n)$ that fixes the last $n-i$ coordinates. 

The {\em normalizer} of the specific unitary subgroup is $NU(i)=U(i)\times U(n-i)$, where by an abuse of notation, $U(n-i)$ is the unitary subgroup that fixes the first $i$ coordinates. Then the quotient group $NU(i)/U(i)$ may be naturally identified with $U(n-i)$. It is usually clear from the context when $U(k)$ is the specific unitary subgroup (fixing the last $n-k$ coordinates) and when it is the quotient group (fixing the first $n-k$ coordinates).

\begin{definition*}
A topological $U(n)$-manifold $M$ is {\em multiaxial}, if any isotropy group is a unitary subgroup, and for any $i>j$, $M^{-i}=M_{-i}-M_{-i-1}$ is a locally flat submanifold in $M_{-j}$.
\end{definition*}

In the definition, the multiaxial manifold $M$ is stratified by $M_{-i}=U(n)M^{U(i)}$, the set of points fixed by some conjugate of $U(i)$. Correspondingly, the orbit space $X=M/U(n)$ is stratified by $X_{-i}=M_{-i}/U(n)$. 

The locally flat assumption can be relaxed. What we really need are some homotopy consequences of this assumption. Specifically, we need the (homotopy) links between adjacent strata to be homotopy spheres, and the pure strata of the (homotopy) links of $M^{-i}$ in $M$ to be connected and simply connected (with the exception that the link of $M^{-1}$ in $M$ can be the circle). Quinn \cite{quinn} showed that such homotopy properties imply that the orbit space is homotopically stratified. Then the pure stratum $M^{-i}=M_{-i}-M_{-i-1}$ is an open manifold that can be completed into a manifold with boundary $U(n)\times_{U(n-i)}(\bar{M}^{U(i)},\pa\bar{M}^{U(i)})$, by deleting (the interior of) regular neighborhoods of lower strata. The pure stratum $X^{-i}=X_{-i}-X_{-i-1}$ is a homology manifold \cite{bfmw}, and can also be completed into a homological manifold with boundary $(\bar{X}^{-i},\pa\bar{X}^{-i})$.

For a multiaxial $U(n)$-manifold $M$, the fixed set $M^{U(i)}$ is a multiaxial $U(n-i)$-manifold, where $U(n-i)=NU(i)/U(i)$ is the quotient group. The following is a kind of ``hereditary property'' for multiaxial manifolds.

\begin{lemma}\label{hereditary}
If $M$ is a multiaxial $U(n)$-manifold, then $M_{-i}/U(n)=M^{U(i)}/U(n-i)$. 
\end{lemma}

The lemma shows that, as far as the orbit space is concerned, the study of a stratum of a multiaxial manifold is the same as the study of a ``smaller'' multiaxial manifold. In particular, if a multiaxial $U(n)$-manifold $M$ does not have free points, then the minimal isotropy groups are conjugate to $U(m)$ for some $m>0$, and the study of the $U(n)$-manifold $M$ is the same as the study of the multiaxial $U(n-m)$-manifold $M^{U(m)}$. Since the $U(n-m)$-action on $M^{U(m)}$ has free points, we may thus always assume the existence of free points without loss of generality. In the setting of multiaxial manifolds modeled on $k\rho_n\oplus j\epsilon$ studied in \cite{davis, dh, dhm}, this means that we may always assume $k\ge n$. We remark that $k\le n$ was always assumed in these earlier works.

Lemma \ref{hereditary} is a consequence of the following two propositions.

\begin{proposition}\label{transitive}
If $H\sub K\sub G=U(n)$ are unitary subgroups, then the $NH$-action on $(G/K)^H$ is transitive. In other words, if $H\sub K$ and $g^{-1}Hg\sub K$, then $g=\nu k$ for some $\nu\in NH$ and $k\in K$. 
\end{proposition}

\begin{proof}
The subgroups $K$ and $H$ consist of the unitary transformations of ${\bb C}^n$ that respectively fix some subspaces $V_K$ and $V_H$. Then $H\sub K$ means $V_K\sub V_H$ and $g^{-1}Hg\sub K$ means $gV_K\sub V_H$. Therefore there is a unitary transformation $\nu$ that preserves $V_H$ and restricts to $g$ on $V_K$. Then $\nu^{-1}g$ preserves $V_K$, so that $\nu^{-1}g\in K$. Moreover, the fact that $\nu$ preserves $V_H$ means that $\nu\in NH$.

The transitivity of the $NH$-action on $(G/K)^H$ means that if $gK\in (G/K)^H$, then $gK=\nu K$ for some $\nu\in NH$. Since $gK\in (G/K)^H$ means $g^{-1}Hg\sub K$, and $gK=\nu K$ means $g= \nu k$ for some $k\in K$, we see that the transitivity is the same as the group theoretical property above.
\end{proof}

\begin{proposition}\label{orbit}
If $G$ acts on a set $M$, such that every pair of isotropy groups satisfy the property in Proposition \ref{transitive}, then $GM^H/G=M^H/NH$ for any isotropy group $H$.
\end{proposition}

\begin{proof}
We always have the natural surjective map $M^H/NH\to GM^H/G$. Over a point in $GM^H/G$ represented by $x\in M^H$, the fibre of the map is $(Gx)^H/NH$.  Therefore the natural map is injective if and only if the action of $NH$ on $(Gx)^H=(G/G_x)^H$ is transitive.
\end{proof}

\section{Homotopy Properties of Multixial $U(n)$-Manifolds}

Although our definition of multiaxial $U(n)$-manifold is more general than those in \cite{davis, dh, dhm} that are modeled on linear representations, many homotopy properties of the linear model are still preserved.

First we consider the (homotopy) link between adjacent strata of the orbit space $X=M/U(n)$ of a multiaxial $U(n)$-manifold $M$. By the link of $X_{-j}$ in $X_{-j+1}$, we really mean the link of the pure stratum $X^{-j}=X_{-j}-X_{-j-1}$ in $X_{-j+1}$ (same for the strata of $M$), and this link may be different along different connected component of $X^{-j}$. So for any $x\in X_{-j}$, we denote by $X_{-j}^x$ the connected component of $X_{-j}$ containing $x$. By the link of $X_{-j}^x$ in $X_{-j+1}$, we really mean the link of $X_{-j}^x-X_{-j-1}$ in $X_{-j+1}$. We also denote by $M^{U(j),x}$ the corresponding connected component of $M^{U(j)}$, so that $X_{-j}^x=M^{U(j),x}/U(n-j)$. 

\begin{lemma}\label{link}
Suppose $X$ is the orbit space of a multiaxial $U(n)$-manifold. For any $x\in X_{-i}$ and $1\le j\le i$, the link of $X_{-j}^x$ in $X_{-j+1}$ is homotopy equivalent to ${\bb C}P^{r_j^x}$, and $r_j^x=r_{j-1}^x+1$. 
\end{lemma}

The lemma paints the following picture of the strata of the links in a (connected) multiaxial $U(n)$-manifold.  For any $x\in X^{-i}$, the stratification near $x$ is given by
\[
X=X_0^x\supset X_{-1}^x\supset \dotsb\supset X_{-i}^x.
\]
The {\em first gap} $r_1^x$ of $x$ depends only on the connected component $X_{-1}^x$ and determines the homotopy type ${\bb C}P^{r_1^x+j-1}$ of the link of $X_{-j}^x$ in $X_{-j+1}^x$. Moreover, we have
\begin{align*}
&\dim M^{U(j-1),x}-\dim M^{U(j),x} \\
&=\dim X_{-j+1}^x+\dim U(n-j+1)-\dim X_{-j}^x-\dim U(n-j) \\
&=\dim {\bb C}P^{r_1^x+j-1}+1+(n-j+1)^2-(n-j)^2 \\
&=2(r_1^x+n).
\end{align*}
The picture also shows that, near a point of $M$ with isotropy group $gU(i)g^{-1}$, $gU(j)g^{-1}$ is the isotropy group of some nearby point for any $1\le j\le i$.

If the multiaxial manifold is modeled on $k\rho_n\oplus j\epsilon$, then the first gap is independent of the connected component, and $r_1=k-n$ in case $k\ge n$. On the other hand, multiaxial $U(1)$-manifolds are just semi-free $S^1$-manifolds, for which any fixed point component has even codimension $2c$, and the first (and the only) gap of the component is $c-1$.

\begin{proof}
The link of $X_{-j}^x$ in $X_{-j+1}$ is the quotient of the link of $M_{-j}$ in $M_{-j+1}$ by the free action of the quotient group $N_{U(j)}U(j-1)/U(j-1)=S^1$. Since $M^{-j}$ is a locally flat submanifold of $M_{-j+1}$, the link is a sphere. The quotient of the sphere by a free $S^1$-action must be homotopy equivalent to a complex projective space ${\bb C}P^{r_j}$.

Let $m_j=\dim M^{U(j),x}$ and $x_j=\dim X_{-j}^x$. By $X_{-j}^x=M^{U(j),x}/U(n-j)$, we have
\[
x_j=\dim M^{U(j),x}-\dim U(n-j)=m_j-(n-j)^2.
\]
Since the link of $X_{-j}^x$ in $X_{-j+1}^x$ is homotopy equivalent to ${\bb C}P^{r_j}$, we also have 
\[
x_{j-1}-x_j=2r_j+1.
\]

Since all the isotropy groups are unitary subgroups, we know $M^{U(j)}=M^{T^j}$ for the maximal torus $T^j$ of $U(j)$. Here $T^j$ is the specific torus group acting by scalar multiplications on the first $j$ coordinates of ${\bb C}^n$. Now we fix $j$ and consider $M$ as a $T^j$-manifold. By the multiaxial assumption, the isotropy groups of the $T^j$-manifold $M$ are the tori that are in one-to-one correspondence with the choices of some coordinates from the first $j$ coordinates of ${\bb C}^n$. The number $j'$ of chosen coordinates is the rank of the isotropy torus. Since all the tori of the same rank $j'$ are conjugate to the specific torus group $T^{j'}$, their fixed point components containing $\tilde{x} \in M^{U(j)}$ (whose image in $X_{-j}$ is $x$) have the same dimension, which is $\dim M^{U(j'),x}=m_{j'}$. 

For the case $j'=j-1$ (corank $1$ in $T^j$), there are $j$ such isotropy tori. By a formula of Borel \cite[Theorem XIII.4.3]{borel}, we have
\[
m_0-m_j=j(m_{j-1}-m_j).
\]
Written in terms of $x_j$, we have
\[
x_0+n^2
=j(x_{j-1}+(n-j+1)^2)-(j-1)(x_j+(n-j)^2),
\]
or
\[
(j-1)^{-1}(x_{j-1}-x_0)-j^{-1}(x_j-x_0)=1.
\]
This gives $x_j-x_0=j(a-j)$ and
\[
x_{j-1}-x_j=2j-1-a.
\]
Combined with $x_{j-1}-x_j=2r_j+1$, we get $r_j=r_{j-1}+1$.
\end{proof}

Next consider the links between any two (not necessarily adjacent) strata of a multiaxial manifold. For multiaxial manifolds locally modeled on $k\rho_n\oplus j\epsilon$, the pure strata of the links are actually homotopy equivalent to Grassmannians and in particular, are simply connected. In the present paper, we only need the following important consequence of the simple connectivity of pure strata of links.

\begin{lemma}\label{fund}
Suppose $X$ is a homotopically stratified space. If all pure strata of the links in $X$ are connected and simply connected, then all strata of the links in $X$ are also connected and simply connected. Moreover, we have
\[
\pi_1(X_{-i}-X_{-j})=\pi_1(X_{-i}),\quad j>i,
\]
and $\pi_1X^{-i}=\pi_1X_{-i}$ in particular.
\end{lemma}

By Lemma \ref{link}, the lemma can be applied to orbit spaces of multiaxial manifolds.

Lemma \ref{fund} follows from Proposition \ref{fund1}. The proposition immediately implies the conclusion $\pi_1(X_{-i}-X_{-j})=\pi_1(X_{-i})$. Then we note that the strata $L_{\alpha}$ of links in $X$ are themselves homotopically stratified spaces, and the links in $L_{\alpha}$ are also the links in $X$. Therefore we may apply the conclusion $\pi_1X^{-i}=\pi_1X_{-i}$ to $L_{\alpha}$ to prove the claim on the strata of links in the lemma. 

The proof of Proposition \ref{fund1} will be based on some well known general observations on the fundamental groups associated to homotopically stratified spaces. In a homotopically stratified space, the neighborhoods of strata are stratified systems of fibrations over the strata. The fundamental groups are related as follows.

\begin{proposition}\label{fibration}
Suppose $E\to X$ is a stratified system of fibrations over a homotopically stratified space $X$. If the fibres are nonempty and connected, then $\pi_1E\to \pi_1X$ is surjective. If the fibres are (nonempty and) connected and simply connected, then $\pi_1E\to \pi_1X$ is an isomorphism.
\end{proposition}

\begin{proof}
If $E\to X$ is a genuine fibration, then the two claims follow from the exact sequence of homotopy groups associated to the fibration.

Inductively, we only need to consider $X=Z\cup_{\pa Z}Y$, where $Y$ is the union of lower strata, $Z$ is the complement of a regular neighborhood of $Y$, and $\pa Z$ is the boundary of a regular neighborhood of $Y$ as well as the boundary of $Z$. Correspondingly, we have $E=E_Z\cup_{E_{\pa Z}}E_Y$, such that $E_Z\to Z$ is a fibration that restricts to the fibration $E_{\pa Z}\to \pa Z$, and $E_Y\to Y$ is a stratified system of fibrations. Then we consider the map
\[
\pi_1E=\pi_1E_Z*_{\pi_1E_{\pa Z}}\pi_1E_Y
\to
\pi_1X=\pi_1Z*_{\pi_1\pa Z}\pi_1Y.
\]
If the fibres of $E\to X$ are connected, then $\pi_1E_Z\to \pi_1Z$ and $\pi_1E_{\pa Z}\to\pi_1\pa Z$ are surjective by the genuine fibration case, and $\pi_1E_Y\to \pi_1Y$ is surjective by induction. Therefore the map $\pi_1E\to \pi_1X$ is surjective. If the fibres of $E\to Z$ are connected and simply connected, then all the maps are isomorphisms, so that $\pi_1E\to \pi_1X$ is an isomorphism.

The proof makes use of van Kampen's theorem, which requires $Y$ to be connected (which further implies that $\pa Z$ is connected). In general, the argument can be carried out by successively adding connected components of $Y$ to $Z$.
\end{proof}

\begin{proposition}\label{fund0}
If $X$ is a homotopically stratified space, such that all pure strata are connected, and all links are not empty, then $X$ is connected. Moreover, if all pure strata are connected and simply connected, and all links are connected, then $X$ is simply connected.
\end{proposition}

We remark that a link $L$ of a stratum $X_{\beta}$ in another stratum $X_{\alpha}$ is stratified, with strata $L_{\gamma}$ corresponding to the strata $X_{\gamma}$ satisfying $X_{\beta}\subsetneq X_{\gamma}\sub X_{\alpha}$. Moreover, the link of $L_{\gamma}$ in $L_{\gamma'}$ is the same as the link of $X_{\gamma}$ in $X_{\gamma'}$. The proposition implies that, if the pure strata of the link between any two strata sandwiched between $X_{\beta}$ and $X_{\alpha}$ are (nonempty and) connected and simply connected, then the link of $X_{\beta}$ in $X_{\alpha}$ is simply connected.

\begin{proof}
If the links are not empty, then any pure stratum is glued to higher pure strata. Therefore the connectivity of all pure strata implies the connectivity of the union, which is the whole $X$.

Now assume that all pure strata are connected and simply connected, and all links are connected. Let $Y$ be a minimum stratum. Then we have decomposition $X=Z\cup_{\pa Z}Y$ similar to the proof of Proposition \ref{fibration}. The complement $Z$ of a regular neighborhood of $Y$ is a stratified space, with the pure strata the same as the pure strata of $X$, except the stratum $Y$. Moreover, the links in $Z$ are the same as the links in $X$. By induction, we may assume that $Z$ (which has one less stratum than $X$) is simply connected. Moreover, $Y$ is a pure stratum and is already assumed to be simply connected. If we know that $\pa Z$ is connected, then we can apply van Kampen's theorem and conclude that $\pi_1X=\pi_1Z*_{\pi_1\pa Z}\pi_1Y$ is trivial.

To see that $\pa Z$ is connected, we note that the base of the fibration $\pa Z\to Y$ is connected. So it is sufficient to show that the fibre $L$ of the fibration is also connected. The fibre is the link $L$ of $Y$ in $X$, and is a stratified space with one less stratum than $X$. Moreover, $L$ has the same link as $X$. Since all pure strata of $X$ are connected, by the first part of the proposition, $L$ is connected. 
\end{proof}

\begin{proposition}\label{fund1}
Suppose $X$ is a homotopically stratified space, and $Y$ is a closed union of strata of $X$. If   for any link between strata of $X$, those pure strata of the link that are not contained in $Y$ are connected and simply connected, then $\pi_1(X-Y)=\pi_1X$.
\end{proposition}

\begin{proof}
We have a decomposition $X=Z\cup_{\pa Z}Y$ similar to that in the proof of Proposition \ref{fibration}. The fibre of the stratified system of fibrations $\pa Z\to Y$ is a stratified space $L_y$ depending on the location of the point $y\in Y$. If $Y^y$ is the pure stratum containing $y$, then the pure strata of $L_y$ are the pure strata of the link of $Y^y$ in $X$ that are not contained in $Y$. By Proposition \ref{fund0} and the remark afterwards, the assumption of the proposition implies that $L_y$ is connected and simply connected. Then we may apply Proposition \ref{fibration} to get $\pi_1\pa Z=\pi_1Y$. Further application of van Kampen's theorem gives us $\pi_1X=\pi_1Z*_{\pi_1\pa Z}\pi_1Y=\pi_1 Z=\pi_1(X-Y)$.
\end{proof}

\section{General Decomposition Theorem}
\label{obstruction}

The homotopy properties in the last section will be used in producing a decomposition theorem for the structure sets of certain stratified spaces. We will use the spectra version of the surgery obstruction, homology and structure set. The equality of spectra really means homotopy equivalence.

\begin{theorem}\label{generalsplit}
Suppose $X=X_0\supset X_{-1}\supset X_{-2}\supset \dotsb$ is a homotopically stratified space, satisfying the following properties:
\begin{enumerate}
\item The homotopy link of $X_{-1}$ in $X$ is homotopy equivalent to ${\bb C}P^r$ with even $r$.
\item The link fibration of $X^{-1}$ in $X$ is orientable, in the sense that the monodromy preserves the fundamental class of the fibre. 
\item For any $i$, the top two pure strata of the link of $X_{-i}$ in $X$ are connected and simply connected.
\end{enumerate}
Then there is a natural homotopy equivalence of surgery obstructions
\[
{\bb L}(X)={\bb L}(X,\rel X_{-2})\oplus {\bb L}(X_{-2}).
\]
Moreover, we have
\[
{\bb L}(X,\rel X_{-2})
={\bb L}(\pi_1X,\pi_1X_{-1}),
\]
and
\[
\pi_1X=\pi_1(X-X_{-1})=\pi_1\bar{X}^0,\quad
\pi_1X_{-1}=\pi_1X^{-1}=\pi_1\pa \bar{X}^0.
\]
\end{theorem}

To prove the theorem, we first establish the following result, which is essentially a reformulation of the periodicity for the classical surgery obstruction \cite[Theorem 9.9]{wa2}.

\begin{proposition}\label{periodicity}
Suppose $X$ is a two-strata space, such that the link fibration of $X_{-1}$ in $X$ is an orientable fibration with fibre homotopy equivalent to ${\bb C}P^r$ with even $r$. Then
\[
{\bb L}(X)
={\bb L}(\pi_1X,\pi_1X_{-1}),\quad
\pi_1X=\pi_1(X-X_{-1}).
\]
\end{proposition}

\begin{proof}
Let $Z$ be the complement of a regular neighborhood of $X_{-1}$ in $X$. Let $E$ be the boundary of $Z$ as well as the boundary of the regular neighborhood. Then $X=Z\cup E\times[0,1]\cup X_{-1}$, and $E\to X_{-1}$ is an orientable fibration with fibre homotopy equivalent to ${\bb C}P^r$.

The surgery obstruction ${\bb L}(X)$ of the two-strata space $X$ fits into a fibration
\[
{\bb L}(E\times[0,1]\cup X_{-1})\to
{\bb L}(X)\to {\bb L}(Z,E),
\]
where the mapping cylinder $E\times[0,1]\cup X_{-1}$ is a regular neighborhood of $X_{-1}$ in $X$ and is a two-strata space with $X_{-1}$ as the lower stratum. The surgery obstruction of the mapping cylinder further fits into a fibration
\[
{\bb L}(E\times[0,1]\cup X_{-1})\xrightarrow{\text{res}} 
{\bb L}(X_{-1})\xrightarrow{\text{trf}} 
{\bb L}(E),
\]
given by the restriction and the transfer. 

Since the fibration $E\to X_{-1}$ is orientable, and the fibre ${\bb C}P^r$ is simply connected, the surgery obstruction groups $\pi_*{\bb L}(X_{-1})$ and $\pi_*{\bb L}(E)$ are described in terms of the same quadratic forms (and formations). Moreover, the effect of the transfer map on surgery obstructions can be computed by the up-down formula of \cite[Theorem 2.1]{lr}. Specifically, the transfer of surgery obstructions is obtained by tensoring with the usual $\pi_1X_{-1}$-equivariant intersection form on the middle homology $H_r{\bb C}P^r={\bb Z}$, where the $\pi_1X_{-1}$-module structure on ${\bb Z}$ comes from the monodromy. Since this tensoring operation induces an isomorphism on the surgery obstruction groups, we conclude that the transfer map is a homotopy equivalence.

We remark that our notion of orientability, as given by the second condition in Theorem \ref{generalsplit}, is weaker than the one in \cite{lr}. Therefore Corollary 2.2 of \cite{lr} cannot be directly applied.

Since the transfer map is a homotopy equivalence, the second fibration implies that ${\bb L}(E\times[0,1]\cup X_{-1})$ is contractible. Then the first fibration further implies that ${\bb L}(X)$ and ${\bb L}(Z,E)$ are homotopy equivalent.

It remains to compute ${\bb L}(Z,E)$. The fibration ${\bb C}P^r\to E\to X_{-1}$ implies $\pi_1E=\pi_1X_{-1}$. By van Kampen's theorem, we have $\pi_1X=\pi_1Z *_{\pi_1E}\pi_1X_{-1}=\pi_1Z=\pi_1(X-X_{-1})$. 
\end{proof}

\begin{proof}[Proof of Theorem \ref{generalsplit}]
Let $Z$ be the complement of a regular neighborhood of $X_{-2}$ in $X$. Let $E$ be the boundary of the regular neighborhood. Then $Z$ and $E$ are two-strata spaces with lower strata $Z_{-1}=Z\cap X_{-1}$ and $E_{-1}=E\cap X_{-1}$. Moreover, $E$ is the boundary of $Z$ in the sense that $E$ has a collar neighborhood in $Z$. We will use $Z$ and $E$ to denote the two-strata spaces, and use $(Z,E)$ to denote the space $Z$ considered as a four-strata space, in which the two-strata of $E$ are also counted.

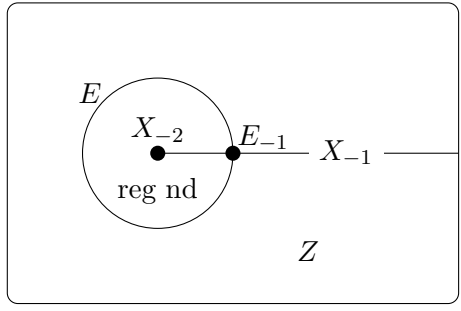
\begin{figure}[h]
\centering
\begin{tikzpicture}[>=latex]

\draw[rounded corners] 
	(-2,-2) rectangle (4,2);

\draw
	(0,0) -- node[black, fill=white, right] {\small $X_{-1}$} (4,0)
	(0,0) circle (1);

\fill
	(0,0) circle (0.1)
	(1,0) circle (0.1);;

\node at (0,0.3) {\small $X_{-2}$};
\node at (2,-1.3) {\small $Z$};
\node at (-0.9,0.8) {\small $E$};
\node at (1.4,0.2) {\small $E_{-1}$};
\node at (0,-0.5) {\small reg nd};

\end{tikzpicture}
\caption{Regular neighborhood of $X_{-2}$ in $X$}
\label{simple}
\end{figure}

Consider a commutative diagram of natural maps of surgery obstructions.
\begin{equation*}\begin{CD}
{\bb L}(Z) @>{\simeq}>> {\bb L}(Z,E) @>>> {\bb L}(E) \\
@VV{\simeq}V @AAA \\
{\bb L}(X,\rel X_{-2})  @>>> {\bb L}(X) @>>> {\bb L}(X_{-2}) 
\end{CD}\end{equation*}
Both horizontal lines are fibrations of spectra. The vertical $\simeq$ is due to the fact that the inclusion $Z\to X-X_{-2}$ of two-strata spaces is a stratified homotopy equivalence. The horizontal $\simeq$ will be a consequence of the fact that ${\bb L}(E)$ is homotopically trivial. The two equivalences give natural splitting to the map ${\bb L}(X,\rel X_{-2})\to{\bb L}(X)$. Then the bottom fibration implies ${\bb L}(X)$ is naturally homotopy equivalent to ${\bb L}(X,\rel X_{-2})\oplus{\bb L}(X_{-2})$.

To see the triviality of ${\bb L}(E)$, we note that the link of $E_{-1}$ in $E$ is the same as the link ${\bb C}P^r$ of $X_{-1}$ in $X$. Therefore we may apply Proposition \ref{periodicity} to $E$ and get 
\[
{\bb L}(E)={\bb L}(\pi_1(E-E_{-1}),\pi_1E_{-1}).
\]
Let $L$ be the link of $X_{-i}$ in $X$, then we have stratified systems of fibrations
\[
L-L_{-1}\to E-E_{-1}\to X_{-2},\quad
L_{-1}-L_{-2}\to E_{-1}\to X_{-2}.
\]
By the third condition, the fibres are always connected and simply connected, and we may apply Proposition \ref{fibration} to get $\pi_1(E-E_{-1})=\pi_1E_{-1}=\pi_1X_{-2}$. By the $\pi$-$\pi$ theorem of the classical surgery theory, we conclude that ${\bb L}(E)$ is homotopically trivial. 

Like $E$, the link of $Z_{-1}$ in $Z$ is also the same as the link ${\bb C}P^r$ of $X_{-1}$ in $X$. Then Proposition \ref{periodicity} tells us 
\[
{\bb L}(X,\rel X_{-2})
={\bb L}(Z)
={\bb L}(\pi_1Z,\pi_1Z_{-1}).
\]
By $Z\simeq X-X_{-2}$, $Z_{-1}\simeq X_{-1}-X_{-2}=X^{-1}$ and Lemma \ref{fund}, we have
\[
\pi_1Z=\pi_1(X-X_{-2})=\pi_1(X-X_{-1})=\pi_1 X,\quad
\pi_1Z_{-1}=\pi_1X^{-1}=\pi_1X_{-1}.
\]
By $X-X_{-1}\simeq \bar{X}^0$ and applying Proposition \ref{fibration} to $\pa \bar{X}^0\to X_{-1}$, which is a stratified system of fibrations with the top strata of the link of $X_{-i}$ in $X$ as fibres, we get
\[
\pi_1Z=\pi_1\bar{X}^0,\quad 
\pi_1Z_{-1}=\pi_1\pa\bar{X}^0.
\qedhere
\]
\end{proof}

The natural splitting for the surgery obstruction in Theorem \ref{generalsplit} induces similar natural splitting for the structure set. 

\begin{theorem}\label{generalsplit2}
Suppose $X=X_0\supset X_{-1}\supset X_{-2}\supset \dotsb$ is a homotopically stratified space, satisfying the following properties:
\begin{enumerate}
\item The homotopy link of $X_{-1}$ in $X$ is homotopy equivalent to ${\bb C}P^r$ with even $r$.
\item The link fibration of $X^{-1}$ in $X$ is orientable as in Theorem \ref{generalsplit}.  
\item The pure strata of all links are connected and simply connected..
\end{enumerate}
Then there is a natural homotopy equivalence of structure sets 
\[
{\bb S}(X)={\bb S}(X,\rel X_{-2})\oplus {\bb S}(X_{-2}).
\]
Moreover, we have
\[
{\bb S}(X,\rel X_{-2})
={\bb S}(\bar{X}^0,\pa \bar{X}^0)
={\bb S}^{\text{\rm alg}}(X,X_{-1}).
\]
\end{theorem}

With no additional work, the third condition can be replaced by the (weaker) third condition in Theorem \ref{generalsplit}, plus the requirement that the fundamental groups $\pi$ of the pure strata of all links satisfy $Wh_i(\pi)=0$ for $i\le 1$.

\begin{proof}
By the topological $h$-cobordism theory \cite{quinn, st}, the third condition implies that the neighborhoods of strata have block bundle structure, the stratified space can be considered as being of the ``PT category'', and the structure set can be computed by the ``unstable surgery fibration'' \cite[Chapter 8]{we1}
\[
{\bb S}(X)\to {\bb H}(X;{\bb L}(\loc X))\to {\bb L}(X).
\]

By Theorem \ref{generalsplit}, we have natural splitting of the surgery spectra
\[
{\bb L}(X)
={\bb L}(X,\rel X_{-2})\oplus {\bb L}(X_{-2})
={\bb L}(\pi_1X,\pi_1X_{-1})\oplus {\bb L}(X_{-2}).
\]
Since the splitting is natural, it can be applied to the coefficient ${\bb L}(\loc X)$ in the homology and induces compatible assembly maps
\[
{\bb H}(X;{\bb L}(\loc(X,\rel X_{-2})))\to {\bb L}(X,\rel X_{-2}),\quad
{\bb H}(X;{\bb L}(\loc X_{-2}))\to {\bb L}(X_{-2}).
\] 
The stratified surgery theory tells us that the homotopy fibre of the first assembly map is the structure set ${\bb S}(X,\rel X_{-2})$. Moreover, we have ${\bb H}(X;{\bb L}(\loc X_{-2}))={\bb H}(X_{-2};{\bb L}(\loc X_{-2}))$ because the coefficient spectrum ${\bb L}(\loc X_{-2})$ is concentrated on $X_{-2}$. Therefore the homotopy fibre of the second assembly map is the structure set ${\bb S}(X_{-2})$. Then we have the decomposition of ${\bb S}(X)$ as stated in the theorem.

It remains to compute ${\bb S}(X,\rel X_{-2})$. The coefficient ${\bb L}(\loc(X,\rel X_{-2}))$ of the homology depends on the location.
\begin{enumerate}
\item At $x\in X^0=X-X_{-1}$, the coefficient is ${\bb L}(D^p)={\bb L}(e)$, where $D^p$ is a ball neighborhood of $x$ in the manifold pure stratum $X^0$. 
\item At $x\in X^{-1}$, the coefficient is ${\bb L}(c{\bb C}P^r\times D^p)$, where $c{\bb C}P^r$ is the cone on the link of $X_{-1}$ in $X$, and $D^p$ is a ball neighborhood of $x$ in the manifold pure stratum $X^{-1}$. Since $r$ is even, the surgery obstruction ${\bb L}(c{\bb C}P^r\times D^p)$ is contractible by Proposition \ref{periodicity}. 
\item At $x\in X_{-2}$, we have $x\in X^{-i}$ for some $i\ge 2$. Let $L$ be the link of $X_{-i}$ in $X$, and let $D^p$ be a ball neighborhood of $x$ in the manifold pure stratum $X^{-i}$. Then the coefficient is 
\begin{align*}
{\bb L}(cL\times D^p,\rel cL_{-2}\times D^p)
&= {\bb L}(cL\times D^p-c\times D^p,\rel cL_{-2}\times D^p-c\times D^p) \\
&= {\bb L}(L\times [0,1]\times D^p,\rel L_{-2}\times [0,1]\times D^p) \\
&= \Omega^{p+1}{\bb L}(L, \rel L_{-2}). 
\end{align*}
We may apply Theorem \ref{generalsplit} to get ${\bb L}(L,\rel L_{-2})={\bb L}(\pi_1L^0,\pi_1L^{-1})$. By the third condition, the pure strata $L^0$ and $L^{-1}$ are connected and simply connected. Therefore the surgery obstruction spectrum is contractible.
\end{enumerate}
Thus the coefficient is the surgery obstruction spectrum ${\bb L}={\bb L}(e)$ on the top pure stratum $X^0=X-X_{-1}$ and is trivial on $X_{-1}$. Therefore the homology is
\[
{\bb H}(X;{\bb L}(\loc(X,\rel X_{-2})))
={\bb H}(X, X_{-1};{\bb L}).
\]
Moreover, Theorem \ref{generalsplit} tells us
\[
{\bb L}(X,\rel X_{-2})
={\bb L}(\pi_1X, \pi_1X_{-1}).
\]
Therefore the homotopy fibre of the assembly map is ${\bb S}^{\text{alg}}(X,X_{-1})$.

By excision, we have ${\bb H}(X, X_{-1};{\bb L})={\bb H}(\bar{X}^0, \pa \bar{X}^0;{\bb L})$. By Theorem \ref{generalsplit}, we also know ${\bb L}(\pi_1X, \pi_1X_{-1})={\bb L}(\pi_1\bar{X}^0, \pi_1\pa \bar{X}^0)$. Therefore the homotopy fibre of the assembly map is also the structure spectrum ${\bb S}(\bar{X}^0, \pa \bar{X}^0)$ of the manifold $(\bar{X}^0, \pa \bar{X}^0)$.
\end{proof}

We note that, in the setup of Theorem \ref{generalsplit2}, the restriction to $X_{-2}$ factors through $X_{-1}$. Then the fact that the restriction ${\bb S}(X)\to{\bb S}(X_{-2})$ is naturally split surjective implies that the restriction ${\bb S}(X_{-1})\to{\bb S}(X_{-2})$ is also naturally split surjective, and we get
\[
{\bb S}(X_{-1})={\bb S}(X_{-1},\rel X_{-2})\oplus {\bb S}(X_{-2}).
\]
Another way of looking at this is that, if a stratified space $X$ is the singular part of a stratified space $Y$ satisfying the conditions of Theorem \ref{generalsplit2}, i.e., $X=Y_{-1}$, then we have the natural splitting
\[
{\bb S}(X)={\bb S}(X,\rel X_{-1})\oplus {\bb S}(X_{-1}).
\]
The following computes ${\bb S}(X,\rel X_{-1})$ for the case relevant to multiaxial manifolds.

\begin{theorem}\label{generalsplit3}
Suppose $X=X_0\supset X_{-1}\supset X_{-2}\supset \dotsb$ is a homotopically stratified space, such that for any $i$, the top pure stratum of the link of $X_{-i}$ in $X$ are connected and simply connected. Then 
\[
{\bb S}(X,\rel X_{-1})
={\bb S}^{\text{\rm alg}}(X).
\]
\end{theorem}

\begin{proof}
Similar to the proof of Theorem \ref{generalsplit2}, the simple connectivity assumption implies that the structure set ${\bb S}(X,\rel X_{-1})$ is the homotopy fibre of the assembly map 
\[
{\bb H}(X;{\bb L}(\loc(X,\rel X_{-1})))\to {\bb L}(X,\rel X_{-1}),
\]
and the coefficient ${\bb L}(\loc(X,\rel X_{-1}))={\bb L}$. We also get $\pi_1(X-X_{-1})=\pi_1X$ from Proposition \ref{fund1}. Therefore the assembly map is ${\bb H}(X;{\bb L})\to {\bb L}(\pi_1X)$, and the homotopy fibre is ${\bb S}^{\text{\rm alg}}(X)$.
\end{proof}

\section{Structure Sets of Multiaxial Actions}
\label{structure}

Let $M$ be a multiaxial $U(n)$-manifold. By Lemma \ref{fund}, the pure strata of links in the orbit space $X=M/U(n)$ are all connected and simply connected. To apply the theorems of Section \ref{obstruction} to $X$, we also need to know the orientability of the link fibration. Since the monodromy map on the fibre ${\bb C}P^r$ comes from the $S^1$-equivariant monodromy map on the homotopy link sphere between the adjacent strata, the monodromy map must be homotopic to the identity. Therefore the link fibration has trivial monodromy and is in particular orientable.

Recall the concept of the first gap defined after the statement of Lemma \ref{link}. The number $r=r_1^x$ depends only on the connected component of the singular part $X_{-1}$. For any connected component $X_{-1}^x$, the number is characterized as the link of $X_{-1}^x$ in $X$ being homotopy to ${\bb C}P^r$. The number is also characterized by the equality $\dim M^{U(j-1),x}-\dim M^{U(j),x}=2(r+n)$.

It is easy to see that Theorem \ref{generalsplit2} remains true in case $X_{-1}$ has several connected components, and perhaps with different ${\bb C}P^r$ for different components, as long as all $r$ are even. Therefore if all the first gaps of a multiaxial $U(n)$-manifold $M$ are even, then we have natural splitting
\[
{\bb S}_{U(n)}(M)
={\bb S}_{U(n)}(M,\rel M_{-2})\oplus {\bb S}_{U(n)}(M_{-2}).
\]

By the computation in Theorem \ref{generalsplit2}, we have
\[
{\bb S}_{U(n)}(M,\rel M_{-2})
={\bb S}(\bar{X}^0,\pa \bar{X}^0)
={\bb S}^{\text{\rm alg}}(X,X_{-1}).
\]
By deleting an equivariant regular neighborhood of $M_{-1}=U(n)\times_{U(n-1)}M^{U(1)}$ from $M$, we get a free $U(n)$-manifold with boundary $(\bar{M}^0,\pa\bar{M}^0)$, and
\[
{\bb S}(\bar{X}^0,\pa \bar{X}^0)
={\bb S}_{U(n)}(\bar{M}^0,\pa\bar{M}^0).
\]

On the other hand, by Lemma \ref{hereditary}, we have ${\bb S}_{U(n)}(M_{-2})={\bb S}_{U(n-2)}(M^{U(2)})$, where $M^{U(2)}$ is a multiaxial $U(n-2)$-manifold. Moreover, Lemma \ref{link} further tells us that, for $x\in M^{U(i)}$, $i>2$, the first gap of $x$ in $M^{U(2)}$ is $r_3^x=r_1^x+2$, where $r_1^x$ is the first gap of $x$ in $M$. This can also be seen from 
\begin{align*}
\dim (M^{U(2)})^{U(j-1),x}-\dim (M^{U(2)})^{U(j),x}
&=\dim M^{U(j-3),x}-\dim M^{U(j-2),x} \\
&=2(r_1^x+n)=2(r_3^x+(n-2)),
\end{align*}
where we use $n-2$ on the right because $M^{U(2)}$ is a multiaxial $U(n-2)$-manifold. The upshot of this is that all the first gaps of $M^{U(2)}$ remain even, and we have further natural splitting
\[
{\bb S}_{U(n)}(M_{-2})
={\bb S}_{U(n-2)}(M^{U(2)}) 
={\bb S}_{U(n-2)}(M^{U(2)},\rel M^{U(2)}_{-2})\oplus {\bb S}_{U(n-2)}(M^{U(2)}_{-2}).
\]
Moreover, we have
\[
{\bb S}_{U(n-2)}(M^{U(2)},\rel M^{U(2)}_{-2})
={\bb S}_{U(n-2)}(\bar{M}^{U(2)},\pa\bar{M}^{U(2)})
={\bb S}^{\text{\rm alg}}(X_{-2},X_{-3}),
\]
and 
\[
{\bb S}_{U(n-2)}(M^{U(2)}_{-2})={\bb S}_{U(n-4)}(M^{U(4)}).
\]

The splitting continues and gives us the general version of part 1 of Theorem \ref{mainth1} in the introduction. The mod $4$ condition on the codimensions is a rephrasement of the even first gap.

\begin{theorem}\label{ssplit}
Suppose $M$ is a multiaxial $U(n)$-manifold, such that the connected components of $M^{U(1)}$ have codimensions $2n$ mod $4$. Then we have natural splitting
\[
{\bb S}_{U(n)}(M)
=\oplus_{i\ge 0}{\bb S}_{U(n-2i)}(\bar{M}^{U(2i)},\pa\bar{M}^{U(2i)})
=\oplus_{i\ge 0}{\bb S}^{\text{\rm alg}}(X_{-2i},X_{-2i-1}).
\]
\end{theorem}

In general, a multiaxial manifold may have even as well as odd first gaps. Denote by $M^{U(1)}_{\text{even}}$ the union of the connected components of $M^{U(1)}$ of dimension $\dim M-2n$ mod $4$. Denote by $M^{U(1)}_{\text{odd}}$ the union of the connected components of $M^{U(1)}$ of dimension $\dim M-2(n+1)$ mod $4$. Then we have
\[
M^{U(i)}=M^{U(i)}_{\text{even}}\cup M^{U(i)}_{\text{odd}},\quad
M^{U(i)}_{\text{even}}=M^{U(i)}\cap M^{U(1)}_{\text{even}},\quad
M^{U(i)}_{\text{odd}}=M^{U(i)}\cap M^{U(1)}_{\text{odd}},
\]
such that the components in $M^{U(i)}_{\text{even}}$ have even first gaps, and the components in $M^{U(i)}_{\text{odd}}$ have odd first gaps. This leads to
\[
M_{-i,\text{even}}=U(n)\times_{U(n-i)}M^{U(i)}_{\text{even}},\quad
M_{-i,\text{odd}}=U(n)\times_{U(n-i)}M^{U(i)}_{\text{odd}}.
\]
We also have the corresponding decompositions 
\[
X_{-i}=X_{-i,\text{even}}\cup X_{-i,\text{odd}},\quad
\bar{M}^{U(i)}=\bar{M}^{U(i)}_{\text{even}}\cup \bar{M}^{U(i)}_{\text{odd}}.
\]

By the same proof as Theorem \ref{ssplit}, we get the same natural splitting for those with even first gaps
\[
{\bb S}_{U(n)}(M)
={\bb S}^{\text{alg}}(X,\rel X_{-2,\text{even}})\oplus {\bb S}^{\text{alg}}(X_{-2,\text{even}})
\]
Here the multiaxial $U(n-2)$-manifold $M^{U(2)}_{\text{even}}$ satisfies the condition of Theorem \ref{ssplit}, so that the second summand can be further split
\[
{\bb S}^{\text{alg}}(X_{-2,\text{even}})
=\oplus_{i\ge 1}{\bb S}^{\text{alg}}(X_{-2i,\text{even}},X_{-2i-1,\text{even}}).
\]
In terms of the multiaxial manifold, this splitting is
\[
{\bb S}_{U(n-2)}(M^{U(2)}_{\text{even}})
=\oplus_{i\ge 1}{\bb S}_{U(n-2i)}(\bar{M}^{U(2i)}_{\text{even}},\pa\bar{M}^{U(2i)}_{\text{even}}).
\]

On the other hand, the first summand 
\[
{\bb S}^{\text{alg}}(X,\rel X_{-2,\text{even}})
={\bb S}_{U(n)}(M,\rel M_{-2,\text{even}}).
\]
Let $N_{\text{even}}$ and $N_{\text{odd}}$ be equivariant neighborhoods of $M_{-1,\text{even}}$ and $M_{-1,\text{odd}}$. Then by the same proof as Theorem \ref{ssplit}, we have
\[
{\bb S}_{U(n)}(M,\rel M_{-2,\text{even}})
={\bb S}_{U(n)}(\overline{M-N_{\text{even}}},\pa N_{\text{even}}).
\]
Combining everything, we get the following decomposition.

\begin{theorem}\label{ssplit3}
Suppose $M$ is a multiaxial $U(n)$-manifold. Then we have natural splitting
\[
{\bb S}_{U(n)}(M)
={\bb S}^{\text{\rm alg}}(X,\rel X_{-2,\text{\rm even}})\oplus \left(\oplus_{i\ge 1}{\bb S}^{\text{\rm alg}}(X_{-2i,\text{\rm even}},X_{-2i-1,\text{\rm even}})\right).
\]
Moreover,
\[
{\bb S}^{\text{\rm alg}}(X,\rel X_{-2,\text{\rm even}})
={\bb S}_{U(n)}(\overline{M-N_{\text{\rm even}}},\pa N_{\text{\rm even}})
\]
and
\[
{\bb S}^{\text{\rm alg}}(X_{-2i,\text{\rm even}},X_{-2i-1,\text{\rm even}})
={\bb S}_{U(n-2i)}(\bar{M}^{U(2i)}_{\text{\rm even}},\pa\bar{M}^{U(2i)}_{\text{\rm even}}).
\]
\end{theorem}

In the theorem above, $U(n-2i)$ acts freely on $\bar{M}^{U(2i)}_{\text{even}}$, and the structure set is about the ordinary manifold $\bar{M}^{U(2i)}_{\text{even}}/U(n-2i)$. 

In the first summand $S_{U(n)}(\overline{M-N_{\text{even}}},\pa N_{\text{\rm even}})$, all the gaps in the multiaxial $U(n)$-manifold $\overline{M-N_{\text{even}}}$ are odd. This leads to the study of multiaxial $U(n)$-manifolds $M$, such that all first gaps are odd. We may use the idea presented before Theorem \ref{generalsplit3}. Suppose $M=W^{U(1)}$ for a multiaxial $U(n+1)$-manifold $W$. Let $Y=W/U(n+1)$ be the orbit space of $W$. Then $X_{-i}=Y_{-i-1}$. By Lemma \ref{link}, for any $x\in X_{-1}=Y_{-2}$, the first gap of $x$ in $Y$ is one less than the first gap of $x$ in $X$. Therefore the first gap of $x$ in $Y$ is even, and the natural splitting of ${\bb S}(Y)$ induces the natural splitting
\[
{\bb S}(X)
={\bb S}(X,\rel X_{-1})\oplus {\bb S}(X_{-1}).
\]
Since the first gap in the $U(n-1)$-manifold $M^{U(1)}$ is one more than the first gap in $M$ and is therefore also even, we may apply Theorem \ref{ssplit} to get further natural splitting
\[
{\bb S}(X_{-1})
=\oplus_{i\ge 1}{\bb S}^{\text{\rm alg}}(X_{-2i+1},X_{-2i}).
\]
On the other hand, by the computation in Theorem \ref{generalsplit3}, the first summand is
\[
{\bb S}(X,\rel X_{-1})
={\bb S}^{\text{\rm alg}}(X).
\]
Then we get the general version of part 2 of Theorem \ref{mainth1} in the introduction.

\begin{theorem}\label{ssplit2}
Suppose $M$ is a multiaxial $U(n)$-manifold, such that such that the connected components of $M^{U(1)}$ have codimensions $2(n+1)$ mod $4$. If $M=W^{U(1)}$ for a multiaxial $U(n+1)$-manifold $W$, then we have natural splitting
\[
{\bb S}_{U(n)}(M)
={\bb S}^{\text{\rm alg}}(X)\oplus\left(\oplus_{i\ge 1}{\bb S}^{\text{\rm alg}}(X_{-2i+1},X_{-2i})\right).
\]
Moreover,
\[
{\bb S}^{\text{\rm alg}}(X_{-2i+1},X_{-2i})
={\bb S}_{U(n-2i+1)}(\bar{M}^{U(2i-1)},\pa\bar{M}^{U(2i-1)}).
\]
\end{theorem}

We remark that, if $M=W^{U(1)}$ and $M$ is connected, then there is only one first gap $r$ in $M$, uniquely determined by
\[
\dim W-\dim M=2(r+n+1).
\]
In case $r$ is odd, there is actually no $M^{U(1)}_{\text{even}}$.

Theorem \ref{mainth2} in the introduction gives another case that ${\bb S}(X_{-1})$ splits off from ${\bb S}(X)$ under the assumption that all the first gaps are odd (but not necessarily equal). The theorem deals with semi-free $S^1$-manifolds, which are the same as multiaxial $U(1)$-manifolds.

\begin{proof}[Proof of Theorem \ref{mainth2}]
The codimensions of $M^{S^1}_0$ and $M^{S^1}_2$ mean that their first gaps are respectively odd and even. Let $N_0$ and $N_2$ be their respective equivariant neighborhoods. Then Theorem \ref{ssplit3} implies
\[
{\bb S}_{S^1}(M)={\bb S}_{S^1}(\overline{M-N_2},\pa N_2).
\]
Next we want to split off the structure set of the fixed point $\overline{M-N_2}^{S^1}=M^{S^1}_0$. 

As we argued at the beginning of the section, the link fibration of $M^{S^1}_0/S^1=M^{S^1}_0$ in $M/S^1$ has trivial monodromy. Moreover, the fibre of this link fibration is homotopy equivalent to ${\bb C}P^r$ for odd (first gap) $r$. By \cite{lr,morgan}, we know that crossing with ${\bb C}P^r$ kills the surgery obstruction. Then the homotopy replacement argument in \cite{cw,cwy} can be applied to show that the natural map 
\[
{\bb S}_{S^1}(\overline{M-N_2},\pa N_2)\to {\bb S}_{S^1}(\overline{M-N_2}^{S^1})={\bb S}_{S^1}(M^{S^1}_0)
\]
is split surjective. Since ${\bb S}_{S^1}(\overline{M-N_2},\pa N_2,\rel \pa N_0)$ is the kernel of the natural map, the theorem is proved.
\end{proof}

\section{Structure Set of Multiaxial Representation Sphere}
\label{repsphere}

Let $\rho_n$ be the defining representation of $U(n)$. Let $\epsilon$ be the real $1$-dimensional trivial representation. Then for any natural number $k$, the unit sphere 
\[
M=S(k\rho_n\oplus j\epsilon)
=S(k\rho_n)*S^{j-1}
\]
of the representation $k\rho_n\oplus j\epsilon$ is a multiaxial $U(n)$-manifold. In this section, we compute the structure set of this representation sphere.

If $k<n$, then $M=U(n)\times_{U(k)}S(k\rho_k\oplus j\epsilon)$, and the problem is reduced to the $U(k)$-representation sphere $S(k\rho_k\oplus j\epsilon)$. Without loss of generality, therefore, we will always assume $k\ge n$ in the subsequent discussion.

The fixed point subsets are
\[
M^{U(i)}=S(k\rho_n^{U(i)}\oplus j\epsilon)
=S(k\rho_{n-i}\oplus j\epsilon)
=S(k\rho_{n-i})*S^{j-1}.
\]
We have
\[
\dim M^{U(i)}=2k(n-i)-1+j,\quad
\dim M^{U(i-1)}-\dim M^{U(i)}=2k.
\]
Therefore the first gap is $k-n$. If $k-n$ is even, then we can use Theorem \ref{ssplit} to compute the structure set. If $k-n$ is odd, then we may use $M=S(k\rho_{n+1}\oplus j\epsilon)^{U(1)}$, where $k-(n+1)$ is even, so that Theorem \ref{ssplit2} can be applied.

We first assume $k-n$ is even and compute the top summand $S^{\text{alg}}(X,X_{-1})$ in the decomposition for $S(X)=S_{U(n)}(S(k\rho_{n}\oplus j\epsilon))$ in Theorem \ref{ssplit}. Since the representation sphere is the link of the origin in the representation space $k\rho_n\oplus j\epsilon={\bb C}^{kn}\oplus{\bb R}^j$, by Lemma \ref{fund}, both $X$ and $X_{-1}$ are connected and simply connected. If the action is neither trivial nor free, then we have $X_{-1}\ne\emptyset$, and the surgery obstruction ${\bb L}(\pi_1X,\pi_1X_{-1})={\bb L}(e,e)$ is trivial. Therefore the top summand is the same as the homology
\[
{\bb S}^{\text{alg}}(X,X_{-1})={\bb H}(X,X_{-1};{\bb L}).
\]

Let 
\[
Z=S(k\rho_n)/U(n),\quad
d=\dim Z=2kn-1-n^2.
\]
Then
\[
(X,X_{-1})=(Z,Z_{-1})*S^{j-1},\quad
\dim X=d+j,
\]
and
\[
S^{\text{alg}}(X,X_{-1})
=\pi_{d+j}{\bb S}^{\text{alg}}(X,X_{-1}) 
=\pi_{d+j}{\bb H}(X,X_{-1};{\bb L}) 
=H_d(Z,Z_{-1};{\bb L}).
\]

\begin{proposition}\label{compute}
If $k\ge n$, then for $Z=S(k\rho_n)/U(n)$, we have
\[
H_{\dim Z}(Z,Z_{-1};{\bb L})
={\bb Z}^{A_{n,k}}\oplus {\bb Z}_2^{B_{n,k}},
\]
where $A_{n,k}$ is the number of $n$-tuples $(\mu_1,\dotsc,\mu_n)$ satisfying
\[
0\le \mu_1\le \dotsb \le \mu_n\le k-n,\quad
\sum\mu_i\text{ is even},
\]
and $B_{n,k}$ is the number of $n$-tuples satisfying
\[
0\le \mu_1\le \dotsb \le \mu_n\le k-n,\quad
\sum\mu_i\text{ is odd}.
\]
\end{proposition}

\begin{proof}
The homology can be computed by a spectral sequence
\[
E^2_{p,q}=H_p(Z,Z_{-1};\pi_q{\bb L}(e))
=\begin{cases}
H_p(Z,Z_{-1};{\bb Z}), & \iif q=0\text{ mod }4, \\
H_p(Z,Z_{-1};{\bb Z}_2), & \iif q=2\text{ mod }4, \\
0, & \iif q\text{ is odd}.
\end{cases}
\]
Since the top pure stratum $Z-Z_{-1}$ is a manifold, by the Poincar\'e duality, we have $H_p(Z,Z_{-1};R)=H^{d-p}(Z-Z_{-1};R)$. The homotopy type of $Z-Z_{-1}$ is well known to be the complex Grassmanian $G(n,k)$. Therefore we have
\[
E^2_{p,q}
=\begin{cases}
H^{d-p}(G(n,k);{\bb Z}), & \iif q=0 \text{ mod }4, \\
H^{d-p}(G(n,k);{\bb Z}_2), & \iif q=2 \text{ mod }4, \\
0, & \iif q\text{ is odd}.
\end{cases}
\]
Using the CW structure of $G(n,k)$ given by the Schubert cells, which are all even dimensional, $E^2_{p,q}$ vanishes when either $q$ or $d-p$ is odd. This implies that all the differentials in $E^2_{p,q}$ vanish, so that the spectral sequence collapses, and we get
\[
H_d(Z,Z_{-1};{\bb L})
= \left(\oplus_{q\le d,\; q=0(4)} H^q(G(n,k);{\bb Z})\right)
\oplus\left(\oplus_{q\le d,\; q=2(4)} H^q(G(n,k);{\bb Z}_2)\right).
\]
Since $G(n,k)$ is a closed manifold, we always have $q\le \dim G(n,k)\le \dim Z=d$. Of course this is nothing but $q\le 2n(k-n)\le 2kn-1-n^2=d$. Therefore the requirement $q\le d$ is automatically satisfied in the summation above, and we have
\[
H_d(Z,Z_{-1};{\bb L})
={\bb Z}^{A_{n,k}}\oplus {\bb Z}_2^{B_{n,k}},
\]
where $A_{n,k}$ is the number of Schubert cells in $G(n,k)$ of dimension $0$ mod $4$, and $B_{n,k}$ is the number of Schubert cells of dimension $2$ mod $4$. The description of $A_{n,k}$ and $B_{n,k}$ in the proposition is the well known numbers of such Schubert cells.
\end{proof}

The unitary group $U(n)$ acts trivially on $S(k\rho_n\oplus j\epsilon)$ only when $n=0$ and $j>0$. In this case, we have $S^{\text{alg}}(X,X_{-1})=S^{\text{alg}}(X)=S(S^{j-1})$. (Here the first $S$ in $S(S^{j-1})$ means the structure set, not the sphere.) By Poincar\'e conjecture, the structure set of the sphere is trivial. This means that we should require $n>0$ in the notation ${\bb Z}^{A_{n,k}}\oplus {\bb Z}_2^{B_{n,k}}$.

The action is free only when $n=1$ and $j=0$. In this case, we have $S^{\text{alg}}(X,X_{-1})=S^{\text{alg}}(X)=S({\bb C}P^{k-1})$. The homology is still ${\bb Z}^{A_{1,k}}\oplus {\bb Z}_2^{B_{1,k}}$. But the surgery obstruction is $L_{2(k-1)}(\pi_1X,\pi_1X_{-1})=L_{2(k-1)}(\pi_1X)=L_{2(k-1)}(e)={\bb Z}$. Here we recall that $k-1=k-n$ is assumed even. Since this piece of surgery obstruction is simply the summand $H^0(G(1,k);{\bb Z})$ in the computation of the homology, this reduces the number of copies of ${\bb Z}$ by $1$. The computation is exactly the fake complex projective space studied in \cite[Section 14C]{wa2}. 

If $k-n$ is even, then Proposition \ref{compute} and the subsequent discussion about the exceptions can be applied to the summands $S^{\text{alg}}(X_{-2i},X_{-2i-1})$ in the decomposition for $S(X)=S_{U(n)}(S(k\rho_{n}\oplus j\epsilon))$ in Theorem \ref{ssplit}, simply by replacing $n$ with $n-2i$. The exception is that, in case $n$ is odd and $j=0$, the $U(1)$-action on $M^{U(n-1)}$ is free, so that $X_{-n}=\emptyset$. The exception happens to the last summand $S^{\text{alg}}(X_{-n+1},X_{-n})=S^{\text{alg}}(X_{-n+1})=S^{\text{alg}}({\bb C}P^{k-1})$, and the number of copies of ${\bb Z}$ is reduced by $1$. This concludes part 1 of Theorem \ref{mainth4}. 

If $k-n$ is odd, then Proposition \ref{compute} can be applied to all summands except the top one in the decomposition for $S(X)$ in Theorem \ref{ssplit2}, simply by replacing $n$ with $n-2i+1$. The exception is that, in case $n$ is even and $j=0$, the last summand is $S^{\text{alg}}(X_{-n+1})=S^{\text{alg}}({\bb C}P^{k-1})$, and the number of copies of ${\bb Z}$ should be reduced by $1$. The top summand $S^{\text{alg}}(X)$ may be computed by the surgery fibration
\[
{\bb S}^{\text{alg}}(X)
\to {\bb H}(X;{\bb L}) 
\to {\bb L}(\pi_1X).
\]
Since $X$ is simply connected, ${\bb L}(\pi_1X)$ is the usual surgery specturm ${\bb L}$, and the assembly map is induced by the map from $X$ to the single point. Therefore 
\[
S^{\text{alg}}(X)
=\tilde{H}_{d+j}(X;{\bb L})
=\begin{cases}
H_d(Z;{\bb L}), & \text{if }j>0, \\
\tilde{H}_d(Z;{\bb L}), & \text{if }j=0.
\end{cases}
\]
The reduced homology is given by Proposition \ref{compute2} of the appendix by Jared Bass. Since $k-n$ is odd, we have
\[
\tilde{H}_d(Z;{\bb L})
={\bb Z}^{A_{n,k-1}}\oplus {\bb Z}_2^{B_{n,k-1}}.
\]
The unreduced homology is modified from the reduced one accordingly to Corollary \ref{compute2-unreduced} of the appendix. This concludes part 2 of Theorem 2 \ref{mainth4}.

\section{Suspension of Multiaxial Representation Sphere}
\label{suspend}

The suspension map is natural with respect to the restrictions to fixed points of unitary subgroups. In other words, the following diagram is commutative.
\begin{equation*}
\begin{CD}
S_{U(n)}(S(k\rho_n\oplus j\epsilon)) @>{*S(\rho_n)}>> 
S_{U(n)}(S((k+1)\rho_n\oplus j\epsilon)) \\
@VV{\text{res}}V @VV{\text{res}}V \\
S_{U(n-i)}(S(k\rho_{n-i}\oplus j\epsilon)) @>{*S(\rho_{n-i})}>> 
S_{U(n-i)}(S((k+1)\rho_{n-i}\oplus j\epsilon))
\end{CD}
\end{equation*}
Since the decomposition of the structure sets of multiaxial manifolds in Section \ref{structure} is obtained from such restrictions, it is tempting to break the suspension map into a direct sum of suspension maps between direct summands. However, such a decomposition is not immediately clear because the parity requirements on $k-n+i$ for the split subjectivity of the restriction maps on the left and right sides are different.

So instead of the (single) suspension, we consider the commutative diagram of the double suspension map.
\begin{equation*}
\begin{CD}
S_{U(n)}(S(k\rho_n\oplus j\epsilon)) @>{*S(2\rho_n)}>> 
S_{U(n)}(S((k+2)\rho_n\oplus j\epsilon)) \\
@VV{\text{res}}V @VV{\text{res}}V \\
S_{U(n-i)}(S(k\rho_{n-i}\oplus j\epsilon)) @>{*S(2\rho_{n-i})}>> 
S_{U(n-i)}(S((k+2)\rho_{n-i}\oplus j\epsilon))
\end{CD}
\end{equation*}
Since the parity requirements for the split subjectivity are the same on both sides, the double suspension map is indeed a direct sum of double suspension maps between direct summands of the respective decompositions of the structure sets. 

We will argue that the double suspension maps between direct summands are injective. This implies that the whole double suspension map is also injective. Since the double suspension map is the composition of two (single) suspension maps, the suspension map is also injective.

To simplify the notations in the discussion, we assume $j=0$. Let
\[
X=S(k\rho_n)/U(n),\quad
Y=S((k+2)\rho_n)/U(n).
\]
We have 
\[
S((k+2)\rho_n)=S(k\rho_n)\times D(2\rho_n)\cup S(2\rho_n),\quad
Y=(S(k\rho_n)\times D(2\rho_n))/U(n)\cup D^3,
\]
and a stratified system of fibrations
\[
p\colon (S(k\rho_n)\times D(2\rho_n))/U(n)
\to X=S(k\rho_n)/U(n).
\]
An element of $S_{U(n)}(S(k\rho_n))$ may be interpreted as a stratified simple homotopy equivalence $f\colon X'\to X$. The pullback of $p$ along $f$ gives a stratified simple homotopy equivalence $Y'\to Y$, which after adding the extra $D^3$ further gives the suspension element of $f$ in $S_{U(n)}(S((k+2)\rho_n))$.

Suppose $k-n$ is even. Then the double suspension map $*S(2\rho_n)$ decomposes into suspension maps between the direct summands
\[
\sigma_i\colon S^{\text{alg}}(X_{-2i},X_{-2i-1})
\to S^{\text{alg}}(Y_{-2i},Y_{-2i-1}).
\]
By the computation in Section \ref{repsphere}, with one exception, the direct summands are the same as the corresponding normal invariants. Therefore we consider the suspension maps on the normal invariants
\[
\sigma_i\colon H_{\dim X_{-2i}}(X_{-2i},X_{-2i-1};{\bb L})
\to H_{\dim Y_{-2i}}(Y_{-2i},Y_{-2i-1};{\bb L}).
\]
The interpretation of the suspension as the pullback of $p$ implies that the suspension of the normal invariants is simply given by the transfer along $p$. On the strata we are interested in, the projection
\[
Y_{-2i}\supset (S(k\rho_{n-2i})\times D(2\rho_{n-2i}))/U(n-2i)
\xrightarrow{p} X_{-2i}=S(k\rho_{n-2i})/U(n-2i)
\]
takes (the orbit of) a $(k+2)$-tuple $\xi=(v_1,\dots,v_k,v_{k+1},v_{k+2})$ of vectors in $\rho_{n-2i}$ and drops the last two vectors $v_{k+1}$ and $v_{k+2}$. Note that $\xi$ is mapped into the pure stratum $X^{-2i}$ if and only if the $k$-tuple $p(\xi)=(v_1,\dots,v_k)$ already spans the whole vector space $\rho_{n-2i}$. This implies that $p^{-1}X^{-2i}\to X^{-2i}$ is a trivial bundle with fibre $D^{2(n-i)}$ (given by the choices $(v_{k+1},v_{k+2})\in D(2\rho_{n-2i})$). By
\[
p^{-1}X^{-2i}=p^{-1}X_{-2i}-p^{-1}X_{-2i-1}
=Y_{-2i}-p^{-1}X_{-2i-1}\cup D^3,\quad
Y_{-2i-1}\sub p^{-1}X_{-2i-1}\cup D^3,
\]
up to excision, the pair $(Y_{-2i},p^{-1}X_{-2i-1}\cup D^3)$ is the same as the Thom space of the trivial disk bundle $p^{-1}X^{-2i}\to X^{-2i}$. The transfer of the normal invariants along this bundle may be identified with the homological Thom isomorphism, and we have a commutative diagram
\begin{equation*}
\begin{CD}
H_{\dim X_{-2i}}(X_{-2i},X_{-2i-1};{\bb L}) @>{\sigma_i}>> 
H_{\dim Y_{-2i}}(Y_{-2i},Y_{-2i-1};{\bb L}) \\
@| @VV{\text{incl}}V \\
H_{\dim X_{-2i}}(X_{-2i},X_{-2i-1};{\bb L})  @>{\text{Thom }\cong}>>
H_{\dim Y_{-2i}}(Y_{-2i},p^{-1}X_{-2i-1}\cup D^3;{\bb L})
\end{CD}
\end{equation*}
The commutative diagram shows that the suspension map $\sigma_i$ is injective.

There is only one exception to the discussion above. In case $n$ is odd (so $k$ is also odd) and $j=0$, the last summand in the decomposition of $S_{U(n)}(S(k\rho_n))$ is $S({\bb C}P^{k-1})$. The double suspension is the usual double suspension map $S({\bb C}P^{k-1})\to S({\bb C}P^{k+1})$, which is well known to be injective. In fact, the structure sets also embed into the corresponding normal invariants, and the injectivity still follows from the Thom isomorphism. This completes the proof of the injectivity of the suspension for the case $k-n$ is even.

Now we turn to the case $k-n$ is odd. The double suspension map $*S(2\rho_n)$ decomposes into suspension maps 
\[
\sigma_i\colon S^{\text{alg}}(X_{-2i+1},X_{-2i})
\to S^{\text{alg}}(Y_{-2i+1},Y_{-2i}),\quad i\ge 1,
\]
and
\[
\sigma_0\colon S^{\text{alg}}(X)
\to S^{\text{alg}}(Y).
\]
The argument for the injectivity of $\sigma_i$ for $i\ge 1$ is the same as the case $k-n$ is even. By the computation in Section \ref{repsphere}, the top summands are the same as the reduced homologies
\[
\sigma_0\colon \tilde{H}_{\dim X}(X;{\bb L})
\to \tilde{H}_{\dim Y}(Y;{\bb L}).
\]
Let $X'^0\sub X^0\sub X$ be those points represented by $k$-tuples of vectors in $\rho_n$, such that the first $(k-1)$ vectors already span the whole vector space $\rho_n$. (This means $r=n$ and $m_n>1$ in Bass' terminology.) Then by the computation of Jared Bass, the map $\tilde{H}_{\dim X}(X;{\bb L})\to \tilde{H}_{\dim X}(X,X-X'^0;{\bb L})$ is injective. On the other hand, the preimage $Y'^0=p^{-1}(X'^0)\sub Y^0\sub Y$ consists of those $(k+2)$-tuples in $\rho_n$, such that the first $(k-1)$ vectors already span the whole vector space $\rho_n$. (This means $r=n$ and $m_n>3$ in Bass' terminology.) Since $Y'^0$ is obtained by adding two vectors $(v_{k+1},v_{k+2})\in D(2\rho_{n})$ to the representatives of points in $X'^0$, the projection $Y'^0\to X'^0$ is a trivial bundle with $D^{2n}$ as fibre. The transfer of the normal invariants along this bundle may be identified with the homological Thom isomorphism, and we have a commutative diagram
\begin{equation*}
\begin{CD}
\tilde{H}_{\dim X}(X;{\bb L}) @>{\sigma_0}>> 
\tilde{H}_{\dim Y}(Y;{\bb L}) \\
@VV{\text{inj}}V @VV{\text{incl}}V \\
H_{\dim X}(X,X-X'^0;{\bb L})  @>{\text{Thom }\cong}>>
H_{\dim Y}(Y,Y-Y'^0;{\bb L})
\end{CD}
\end{equation*}
The commutative diagram shows that the suspension map $\sigma_0$ is injective.

Again there is only one exception to the discussion. In case $n$ is even (so $k$ is odd), and $j=0$, the last summand in the decomposition of $S_{U(n)}(S(k\rho_n))$ is $S({\bb C}P^{k-1})$. The double suspension on this summand is injective, just like the exceptional case when $k-n$ is even. This completes the proof of the injectivity of the suspension for the case $k-n$ is odd.

\section{Multiaxial $Sp(n)$-manifolds}
\label{quat}

The symplectic group $Sp(n)$ consists of $n\times n$ quaternionic matrices that preserve the standard hermitian form on ${\bb H}^n$,
\[
\langle x,y\rangle
=\bar{x}_1y_1+\bar{x}_2y_2+\dotsb+\bar{x}_ny_n.
\]
A symplectic subgroup associated to a quaternionic subspace of ${\bb H}^n$ consists of the quaternionic matrices that preserve the standard hermitian form and fix the quaternionic subspace. Any symplectic subgroup is conjugate to a specific symplectic subgroup $Sp(i)$ associated to the specific subspace $0\oplus {\bb H}^{n-i}$.

We call an $Sp(n)$-manifold multiaxial, if any isotropy group is a  symplectic subgroup, and lower strata are locally flat submanifolds of higher strata. All our discussion about multiaxial $U(n)$-manifolds can be extended to multiaxial $Sp(n)$-manifolds.

The role played by $U(1)=S^1$ is replaced by $Sp(1)=S^3$, the group of quaternions of unit length. If $S^3$ acts freely on a sphere, then the dimension of the sphere is $3$ mod $4$, and the quotient is homotopy equivalent to ${\bb H}P^r$. In analogy to the unitary case, we have $M^{Sp(j)}=M^{T^j}$ for the maximal torus $T^j$ of $Sp(j)$, and all such maximal tori for the given $j$ are conjugate in $Sp(n)$. Hence the proof of Lemma \ref{link} using the Borel formula remains valid, and we get a quaternionic version of the formula for the first gap,  
\[
\dim M^{Sp(j-1),x}-\dim M^{Sp(j),x} 
=4(r_1^x+n).
\]
Since ${\bb H}P^r$ is always connected and simply connected, Lemma \ref{fund} can also be applied to multiaxial $Sp(n)$-manifolds. 

The key reasons behind the results in Section \ref{obstruction} is that for even $r$, ${\bb C}P^r$ is a manifold of signature one, which makes the surgery transfer an equivalence, even after taking account of the monodromy. This remains valid with ${\bb H}P^r$ in place of ${\bb C}P^r$, so that all the results in Section \ref{obstruction} still hold. 

The key reason that we can apply the results in Section \ref{obstruction} to multiaxial $U(n)$-manifolds is that the fibres of the link fibration between adjacent strata are homotopy equivalent to ${\bb C}P^r$, and the link fibration has trivial monodromy and is therefore orientable. Since the same reason remains valid for multiaxial $Sp(n)$-manifolds, the splitting theorems in Section \ref{structure} can be extended.

\begin{theorem}
Suppose $M$ is a multiaxial $Sp(n)$-manifold, such that the connected components of $M^{Sp(1)}$ have codimensions $4n$ mod $8$. Then we have natural splitting
\[
{\bb S}_{Sp(n)}(M)
=\oplus_{i\ge 0}{\bb S}_{Sp(n-2i)}(\bar{M}^{Sp(2i)},\pa\bar{M}^{Sp(2i)})
=\oplus_{i\ge 0}{\bb S}^{\text{\rm alg}}(X_{-2i},X_{-2i-1}).
\]
\end{theorem}

\begin{theorem}
Suppose $M$ is a multiaxial $Sp(n)$-manifold, such that the connected components of $M^{Sp(1)}$ have codimensions $4(n+1)$ mod $8$. If $M=W^{Sp(1)}$ for a multiaxial $Sp(n+1)$-manifold $W$, then we have natural splitting
\[
{\bb S}_{Sp(n)}(M)
={\bb S}^{\text{\rm alg}}(X)\oplus\left(\oplus_{i\ge 1}{\bb S}^{\text{\rm alg}}(X_{-2i+1},X_{-2i})\right).
\]
Moreover,
\[
{\bb S}^{\text{\rm alg}}(X_{-2i+1},X_{-2i})
={\bb S}_{Sp(n-2i+1)}(\bar{M}^{Sp(2i-1)},\pa\bar{M}^{Sp(2i-1)}).
\]
\end{theorem}

Theorem \ref{ssplit3} can be extended. The proof of Theorem \ref{mainth2} at the end of Section \ref{structure} can also be extended, in view of the fact that the signature of ${\bb H}P^r$ is zero for odd $r$. So we have the quaternonic version of Theorem \ref{mainth2}. 

\begin{theorem}
Suppose the quaternionic sphere $S^3$ acts semifreely on a topological manifold $M$, such that the fixed points $M^{S^3}$ is a locally flat submanifold. Let $M_0^{S^3}$ and $M_2^{S^3}$ be the unions of those connected components of $M^{S^3}$ that are, respectively, of codimensions $0$ mod $8$ and $4$ mod $8$. Let $N$ be the complement of (the interior of) an equivariant tube neighborhood of $M^{S^3}$, with boundaries $\pa_0N$ and $\pa_2N$ corresponding to the two parts of the fixed points. Then
\[
S_{S^3}(M)=S(M_0^{S^3})\oplus S(N/S^3,\pa_2N/S^3,\rel \pa_0N/S^3).
\]
\end{theorem}

We can also compute the structure sets of multiaxial $Sp(n)$-representation spheres. The dimensions of the Schubert cells of quaternionic Grassmannians $G_{\bb H}(n,k)$ are multiples of $4$, so that the analogue of Proposition \ref{compute} gives copies of $L_{4i}(e)={\bb Z}$, regardless of the parity. Since the total number of Schubert cells in $G_{\bb H}(n,k)$ is $A_{n,k}+B_{n,k}=\binom{k}{n}$, we have
\[
H_d(S(k\rho_n)/Sp(n),S(k\rho_n)_{-1}/Sp(n);{\bb L})
={\bb Z}^{\binom{k}{n}},\quad k\ge n,
\]
where
\[
d=\dim S(k\rho_n)/Sp(n)=4kn-1-n(2n+1).
\]

On the other hand, the CW structure by Jared Bass can also be applied to the orbit space $S(k\rho_n)/Sp(n)$. The reason is that for complex matrices, the unique representative by row echelon form is a consequence of the fact that $GL(n,{\bb C})=U(n)N$, where $U(n)$ is the maximal compact subgroup of the semisimple Lie group $SL(n,{\bb C})$ and $N$ is the upper triangular matrix with positive diagonal entries. This is a special example of the Iwasawa decomposition. When the decomposition is applied to the semisimple Lie group $SL(n,{\bb H})$, for which $Sp(n)$ is the maximal compact subgroup, we get $GL(n,{\bb H})=Sp(n)N$. Therefore the orbit space $S(k\rho_n)/Sp(n)$ has cells $B(m_1,\dotsc,m_r)$ similar to the orbit space $S(k\rho_n)/U(n)$, except that
\[
\dim B(m_1,\dotsc,m_r)=4(m_1+\dotsb+m_r)-3r-1.
\] 
This leads to the analogue of Proposition \ref{compute2}
\[
\tilde{H}_d(S(k\rho_n)/Sp(n);{\bb L})
={\bb Z}^{\binom{k-1}n},\quad
k\ge n.
\]
For the case $k-n$ is odd, this is the top summand 
\[
S^{\text{alg}}(S(k\rho_n)/Sp(n))
=\tilde{H}_d(S(k\rho_n)/Sp(n);{\bb L})
\]
in the decomposition of the structure set $S_{Sp(n)}(S(k\rho_n))$. If $k-n$ is odd and $j>0$, then the top summand is
\begin{align*}
S^{\text{alg}}((k\rho_n\oplus j\epsilon)/Sp(n))
&=\tilde{H}_{d+j}(X;{\bb L})
=H_d(S(k\rho_n)/Sp(n);{\bb L}) \\
&=\tilde{H}_d(S(k\rho_n)/Sp(n);{\bb L})\oplus H_0(Z;\pi_d{\bb L}).
\end{align*}
The extra homology at the base point is
\[
H_0(Z;\pi_{4kn-1-n(2n+1)}{\bb L})
=L_{4kn-1-n(2n+1)}(e)
=\begin{cases}
{\bb Z}, & \iif n=1\text{ mod }4, \\
{\bb Z}_2, & \iif n=3\text{ mod }4, \\
0, & \iif n \text{ is even}.
\end{cases}
\]

Finally, we need to consider the case the last summand in the decomposition is $S({\bb H}P^{k-1})$, which happens when $k,n$ odd and $j=0$, or $k$ odd, $n$ even and $j=0$. In this case, the number of copies of ${\bb Z}$ should be reduced by $1$. 

In summary, the quaternionic analogue of Theorem \ref{mainth4} is the following.

\begin{theorem}\label{structure1b}
Suppose $k\ge n$ and $\rho_n$ is the canonical representation of $Sp(n)$. 
\begin{enumerate}
\item If $k-n$ is even, then
\[
S_{Sp(n)}(S(k\rho_n\oplus j\epsilon))
={\bb Z}^{\sum_{0\le 2i<n}\binom{k}{n-2i}},
\]
with the only exception that there is one less ${\bb Z}$ in case $n$ is odd and $j=0$.
\item If $k-n$ is odd, then
\[
S_{Sp(n)}(S(k\rho_n\oplus j\epsilon))
={\bb Z}^{\binom{k-1}{n}+\sum_{0\le 2i-1<n}\binom{k}{n-2i+1}},
\]
with the following exceptions: (i) There is one less ${\bb Z}$ in case $n$ is even and $j=0$; (ii) There is one more ${\bb Z}$ in case $n=1$ mod $4$ and $j>0$; (iii) There is one more ${\bb Z}_2$ in case $n=3$ mod $4$ and $j>0$.
\end{enumerate}
\end{theorem}

Finally, the discussion on the suspension 
\[
*S(\rho_n)\colon
S_{Sp(n)}(S(k\rho_n\oplus j\epsilon))
\to S_{Sp(n)}(S((k+1)\rho_n\oplus j\epsilon))
\]
can be carried out just like Section \ref{suspend} and we conclude the injectivity of the suspension.

\appendix
\section{Homology of Quotients of Multiaxial Representation Spheres, by Jared Bass}

Following earlier notation, we say
\[
Z = S(k\rho_n)/U(n), \quad
d = \dim Z = 2kn - 1 - n^2.
\]
Through an explicit CW decomposition, we will compute the reduced homology $\tilde{H}_{d}(Z;{\bb L})$.

\begin{proposition}\label{compute2}
If $k\ge n$, then for $Z=S(k\rho_n)/U(n)$, we have
\[
\tilde{H}_{\dim Z}(Z;{\bb L})
={\bb Z}^{a_{n,k}}\oplus {\bb Z}_2^{b_{n,k}},
\]
where $a_{n,k}$ is the number of $n$-tuples $(\mu_1,\dotsc,\mu_n)$ satisfying
\[
0\le \mu_1\le\dotsb\le \mu_n\le k-n-1,\quad
\sum\mu_i+kn\text{ is even},
\]
and $b_{n,k}$ is the number of $n$-tuples satisfying
\[
0\le \mu_1\le\dotsb\le \mu_n\le k-n-1,\quad
\sum\mu_i+kn\text{ is odd}.
\]
\end{proposition}

In the case $k-n$ is odd, which is what we are really interested in, we note that $\sum\mu_i+kn$ and $\sum\mu_i$ have the same parity, so that $a_{n,k}=A_{n,k-1}$ and $b_{n,k}=B_{n,k-1}$ from Proposition \ref{compute}. In the case $k-n$ is even,  $\sum\mu_i+kn$ and $\sum\mu_i+n$ have the same parity.

\begin{proof}
An element in $S(k\rho_n)$ is a $k$-tuple $\xi=(v_1,\dotsc,v_k)$ of vectors in $\rho_n$ satisfying $\|\xi\|^2=\|v_1\|^2+\dotsb+\|v_k\|^2=1$, with the $U(n)$-action $g\xi=(gv_1,\dotsc,gv_k)$. We may regard $\xi$ as a complex $k\times n$-matrix. We claim that we can find a unique representative for $\xi$ of in the row echelon form
\[
\bar{\xi}=
\begin{bmatrix}
\quad & \lambda_1 & \dotsb & *         & \dotsb & *         & \dotsb & *        & \dotsb \\
 &           &        & \lambda_2 & \dotsb & *         & \dotsb & *        & \dotsb \\
 &           &        &           &        & \lambda_3 & \dotsb & *        & \dotsb \\
 &           &        &           &        &           & \ddots & \vdots   & \dotsb \\
 &           &        &           &        &           &        & \lambda_r & \dotsb \\
 &           &        &           &        &           &        &          &
\end{bmatrix},
\]
where the empty spaces are occupied by $0$, $*$ and dots mean complex numbers, $\lambda_i>0$, and the total length of all the entries is $1$, as it was for $\xi$.  To get $\bar{\xi}$, apply the Gram-Schmidt process to the columns of $\xi$ to obtain an orthonormal basis for $\bb{C}^n$ (adding extra vectors if necessary).  If we then apply to $\xi$ the unitary matrix taking this new basis to the standard basis, we get $\bar{\xi}$ as desired.  The orbit space $Z$ is the collection of all matrices $\bar{\xi}$ of the above form.

If $\lambda_j$ appears $m_j$ places from the right end of the matrix (i.e., $\lambda_j$ lies in the $k-m_j+1$ column), then we say that the matrix has {\em shape} $(m_1,\dotsc,m_r)$.  Note that $r$ is the rank of the matrix $\xi$.  For any $r\le n$, $k\ge m_1>\dotsb>m_r>0$, all $\bar{\xi}$ of the shape $(m_1,\dotsc,m_r)$ form a cell $B(m_1,\dotsc,m_r)$ of dimension
\[
\dim B(m_1,\dotsc,m_r)=2(m_1+\dotsb+m_r)-r-1.
\]
Geometrically, the cell is the subset of a sphere of the above dimension determined by $r$ coordinates being nonnegative.  The boundary of this cell consists of those shapes $(m'_1,\dotsc,m'_{r'})$ satisfying $r'\le r$ and $m'_i\le m_i$, with at least one inequality being strict. In homological computation, only those shapes of one dimension less matter. This only occurs when
\[
m_r=1,\quad
r'=r-1,\quad
m'_i=m_i \text{ for }1\le i<r.
\]
Therefore, the only nontrivial boundary map of the cellular chain complex is
\[
\pa B(m_1,\dotsc,m_{r-1},1)=B(m_1,\dotsc,m_{r-1}).
\]
The homology is then freely generated by the shapes that are neither $(m_1,\dotsc,m_{r-1},1)$ nor $(m_1,\dotsc,m_{r-1})$ in the equality above. These are exactly the shapes satisfying $r=n$ (meaning $\xi$ has full rank) and $m_n>1$, and the shape $(1)$ (meaning $r=1$ and $m_1=1$). The shape $(1)$ is the base point of $Z$.

The reduced homology $\tilde{H}_*(Z;{\bb L})$ is the limit of a spectral sequence with
\[
E_2^{p,q}=\tilde{H}_p(Z;\pi_q{\bb L})
=\begin{cases}
\tilde{H}_p(Z;{\bb Z}), &\text{if } q=0\text{ mod }4, \\
\tilde{H}_p(Z;{\bb Z}_2), &\text{if } q=2\text{ mod }4, \\
0, &\text{if } q\text{ is odd}.
\end{cases}
\]
Note that the reduced homology $\tilde{H}_pZ$ is freely generated by shapes satisfying $r=n$ and $m_n>1$. Since the dimensions of such cells have the same parity as $n+1$, $\tilde{H}_pZ$ is nontrivial only if $p$ has the same parity as $n+1$. This implies that $E_2^{p,q}$ already collapses and
\[
\tilde{H}_d(Z;{\bb L})
=(\oplus_{q=0(4)}\tilde{H}_{d-q}(Z;{\bb Z}))\oplus(\oplus_{q=2(4)}\tilde{H}_{d-q}(Z;{\bb Z}_2)).
\]
We have
\[
\oplus_{q=0(4)}\tilde{H}_{d-q}(Z;{\bb Z})
={\bb Z}^{a_{n,k}},
\]
where $a_{n,k}$ is the number of shapes $(m_1,\dotsc,m_n)$ satisfying
\[
m_n>1, \quad
2(m_1+\dotsb+m_n)-n-1=d = 2kn - 1 - n^2\text{ mod }4.
\]
Let $\mu_i=m_{n-i+1}-(i+1)$, so this condition can be interpreted in terms of the nondecreasing sequence of nonnegative integers $(\mu_1,\dotsc,\mu_n)$, as in the statement of the proposition.  Through a similar computation we get the description of $b_{n,k}$ for the case $q=2$ mod $4$.
\end{proof}

For the unreduced homology $H_d(Z;{\bb L})$, we also need to consider the basepoint.  So we need to further take the direct sum with the homology at the base, $H_0(Z;\pi_d{\bb L})=L_d(e)$.  In our case of interest, when $k-n$ is odd, we have $d=n^2+1$ mod $4$. This yields the following.

\begin{corollary} \label{compute2-unreduced}
For $k-n$ odd, the unreduced homology $H_{\dim Z}(Z;{\bb L})$ is given by Proposition \ref{compute2} with an additional summand of
\[
H_0(Z;\pi_d{\bb L})
=L_d(e)
=\begin{cases}
{\bb Z}_2, & \text{if $n$ is odd}, \\
0, & \text{if $n$ is even}.
\end{cases}
\]
\end{corollary}

\end{document}